\documentclass{tsp-short}
\newtheorem{thm}{Theorem}[section]
\newtheorem{lem}{Lemma}[section]

\newtheorem{prp}{Proposition}[section]
\theoremstyle{remark}
\newtheorem{rem}{Remark}[section]
\theoremstyle{definition}
\newtheorem{dfn}{Definition}[section]

\begin{document}

\title
[Constructing stochastic flows of kernels]
{Constructing stochastic flows of kernels}

\author{Georgii Riabov}
\address{Institute of Mathematics of NAS of Ukraine}
\email{ryabov.george@gmail.com}




\begin{abstract}
In the paper we suggest a new construction of stochastic flows of kernels in a locally compact separable metric space $M$. Starting from a consistent sequence of Feller transtition function $(\mathsf{P}^{(n)}: n\geq 1)$ on $M$ we prove existence of a stochastic flow of kernels $K=(K_{s,t}: -\infty<s\leq t<\infty)$ in $M,$ such that distributions of $n$-point motions  of $K$ are determined by $\mathsf{P}^{(n)}.$ Presented construction allows to find a single idempotent measurable presentation $\mathfrak{p}$ of distributions of all kernels $K_{s,t}$ from the flow, and to construct a flow that is invariant under $\mathfrak{p}$ and is jointly measurable in all arguments.
\end{abstract}

\maketitle

\section{Introduction}

Stochastic flows of kernels appear naturally as solutions to stochastic differential equations (SDE's) in the absence of strong uniqueness.  Following fundamental works of Y. Le Jan and O. Raimond \cite{LeJan_Raimond_1,LeJan_Raimond_2}, by a stochastic flow of kernels we understand a family $(K_{s,t}:-\infty<s\leq t<\infty)$ of random probability transition kernels on a locally compact separable metric space $M$  that satisfy the  evolutionary property $K_{r,s}K_{s,t}=K_{r,t},$  $K_{s,s}(x)=\delta_x,$ $r\leq s\leq t$ (equalities must be understood in a proper sense that is explained below), have independent and homogeneous increments (if $t_1\leq t_2\leq \ldots \leq t_n,$ then $K_{t_1,t_2},\ldots,K_{t_{n-1},t_n}$ are independent; the distribution of $K_{s,t}$ depends only on $t-s$) and satisfy a variant of the Feller condition. Precise definition of a stochastic flow of kernels is given in Section \ref{sec:def}. 

One of the simplest examples of an SDE for which strong uniqueness fails is the Tanaka equation on $\mathbb{R}$
\begin{equation}
\label{eq:Tanaka_SDE}
dX_t=\mbox{sign}(X_t)d\mathrm{B}_t,
\end{equation}
where $(\mathrm{B}_t: t\in \mathbb{R})$ is the standard Brownian motion on $\mathbb{R}$ \cite[Ch. IV, \S 1]{Watanabe_Ikeda}. Obviously, the solution $X$ of \eqref{eq:Tanaka_SDE} follows the trajectory of $\mathrm{B}$ when it is strictly  positive, and follows the trajectory of $-\mathrm{B}$ when it is strictly negative. The reason for non-existence of a strong solution is that once the solution $X$ reaches zero, it can randomly choose which excursion to follow: the excursion of $\mathrm{B}$ or the excursion of $-\mathrm{B}$. A natural extension of the Tanaka equation to kernels was suggested in \cite{LeJan_Raimond_ALEA} in the form
\begin{equation}
\label{eq:Tanaka_kernels}
K_{s,t}f(x)=f(x)+\int^t_s K_{s,u}\left(f'\mbox{sign}\right)(x)d\mathrm{B}(u)+\frac{1}{2}\int^t_s K_{s,u}(f'')(x)du, \ t\geq s,
\end{equation}
where $f$ is an arbitrary twice continuously differentiable function on $\mathbb{R}$ with compact support.
If kernels $K_{s,t}$ are given by random mappings of $\varphi_{s,t}:\mathbb{R}\to \mathbb{R},$ i.e. $K_{s,t}(x)=\delta_{\varphi_{s,t}(x)},$ then the equation \eqref{eq:Tanaka_kernels} is a consequence of the It\^o formula. However, there are kernel solutions to \eqref{eq:Tanaka_kernels} that are not given by random mappings.  In \cite{LeJan_Raimond_ALEA} it was proved that all solutions of \eqref{eq:Tanaka_kernels} are in one-to-one correspondence with probability measures $m$ on $[0,1]$ with mean $\frac{1}{2},$ where $m$ is the law of $K_{0,t}(0,[0,\infty)).$ An amount of similar results for large classes of SDE's on manifolds and metric graphs were obtained in \cite{Hajri_WBM, Hajri_Caglar_Arnaudon, Hajri_Raimond_SPA,  LeJan_Raimond_IBVF, LeJan_Raimond_circle, LeJan_Raimond}. Stochastic flows of kernels with Brownian $n$-point motions were studied in \cite{Howitt_Warren, Schertzer_Sun_Swart, Warren}.

In \cite{LeJan_Raimond_1, LeJan_Raimond_2} it was shown that to any sequence $(\mathsf{P}^{(n)}:n\geq 1)$ of consistent Feller transition functions (where $(\mathsf{P}^{(n)}_t:t\geq 0)$ is a Feller transition function on $M^n$) there corresponds a stochastic flow of kernels $(K_{s,t}:-\infty<s\leq t<\infty)$ such that for all $n\geq 1,$ $t\geq 0,$ $x\in M^n$
\begin{equation}
\label{eq:FTF_to_SFK}
\mathsf{P}^{(n)}_t f(x)=\mathsf{E}\left[\int_{M^n} f(y)\left(\otimes^n_{i=1} K_{0,t}(x_i)\right)(dy)\right],
\end{equation}
where $f$ is an arbitrary continuous function on $M^n$ that vanishes at infinity. Consistency of transition functions means that transition kernels $\mathsf{P}^{(n)}_t(x)$ behave properly under permutations of components of $x\in M^n$ and define transition kernels $\mathsf{P}^{(k)}_t(y)$ for all $k<n$ and $y\in M^k.$ This result extends results of \cite{Darling, LeJan_Raimond_1,LeJan_Raimond_2} on existence of stochastic flows of mappings. In \cite{Darling} it was proved that to any sequence $(\mathsf{P}^{(n)}:n\geq 1)$ of consistent transition functions with additional  property that $\mathsf{P}^{(2)}_t((x,x))$ is concentrated on a diagonal of $M^2$ (coalescencing property) there corresponds a stochastic flow of random mappings $(\varphi_{s,t}:-\infty <s\leq t<\infty)$ of $M$ such that  for all $n\geq 1,$ $t\geq 0,$ $x\in M^n$
$$
\mathsf{P}^{(n)}_t f(x)=\mathsf{E} f\left(\varphi_{0,t}(x_1),\ldots,\varphi_{0,t}(x_n)\right),
$$
where $f$ is an arbitrary continuous function on $M^n$ that vanishes at infinity. In the construction of \cite{Darling} the evolutionary property $\varphi_{s,t}\circ \varphi_{r,s}=\varphi_{r,t},$ $r\leq s\leq t,$ holds without exceptions in $r,s,t,\omega$,  for any $t_1\leq t_2\leq \ldots \leq t_n$ mappings  $\varphi_{t_1,t_2},\ldots,\varphi_{t_{n-1},t_n}$ are independent, and the distribution of $\varphi_{s,t}$ depends only on $t-s.$ Howeve, in this construction the measurability of $\varphi_{s,t}(x)$ in any of the variables $s,t$ or $x$ is absent and only measurability in $x$ can be achieved under rather strong restrictions on transition functions $(\mathsf{P}^{(n)}: n\geq 1).$ This limits the applicability of results of \cite{Darling} in the context of equations like \eqref{eq:Tanaka_kernels}. To overcome the issue, in \cite{LeJan_Raimond_1, LeJan_Raimond_2}  the  Feller property of $P^{(n)}$ is assumed and the definition of a stochastic flow is modified. Namely, a stochastic flow of mappings is a family $(\varphi_{s,t}:-\infty <s\leq t<\infty)$ of random elements in the space of measurable mappings of  $M$ (equipped with the cylindrical $\sigma$-field) that satisfies  a variant of the Feller property and for which  the evolutionary property is understood as follows:

for all $r\leq s\leq t$ and $x\in M$ with probability $1$
\begin{equation}
\label{eq:EP_MP_SFM}
\varphi_{r,t}(x)=\mathcal{J}_{t-s}\left(\varphi_{s,t}\right)\circ \varphi_{r,s}(x),
\end{equation}
where $\mathcal{J}_{t-s}$ is a measurable presentation of the distribution of $\varphi_{s,t}$ in the space of measurable mappings of $M$. The usage of a measurable presentation $\mathcal{J}_{t-s}$ together with a variant of the Feller property for  $\varphi$ allows to settle a one-to-one correspondence between stochastic flows of mappings and  coalescing sequences of consistent Feller transition functions. Similarly, the evolutionary property for stochastic flows of kernels in \cite{LeJan_Raimond_1, LeJan_Raimond_2} is understood as follows:

for all $r\leq s\leq t$ and $x\in M$ with probability $1$
\begin{equation}
\label{eq:EP_MP_SFK}
K_{r,t}(x)=K_{r,s}\mathfrak{p}_{t-s}\left(K_{s,t}\right)(x),
\end{equation}
where $\mathfrak{p}_{t-s}$ is a measurable presentation of the distribution of $K_{s,t}$ in the space of kernels on $M$ (see Section  \ref{sec:def} for more details). 

Presences of $\mathcal{J}_{t-s}$ in \eqref{eq:EP_MP_SFM} and  of   $\mathfrak{p}_{t-s}$ in \eqref{eq:EP_MP_SFK} do not look natural. However, they are necessary due to two reasons at least. Firstly, the convolution of kernels is in general a non-measurable operation and it is not clear how to define convolution of two independent random kernels in a measurable way. Secondly, the presence of $\mathfrak{p}_{t-s}$ in \eqref{eq:EP_MP_SFK} allows to show that functions $\mathsf{P}^{(n)}_t(x)$ defined in \eqref{eq:FTF_to_SFK} are actually transition functions. In \cite{LeJan_Raimond_1, LeJan_Raimond_2} a stochastic flow of mappings $(\varphi_{s,t}:-\infty<s\leq t<\infty)$ was constructed in such a way that equalities $\mathcal{J}_{t-s}(\varphi_{s,t}(\omega))=\varphi_{s,t}(\omega)$ were satisfied without exceptions in $s,t,\omega$. The same result for flows of kernels was absent. The reason is that in \cite{LeJan_Raimond_1, LeJan_Raimond_2} flow of kernels is constructed from a certain stochastic flow of measure-valued mappings, and the procedure that produces the flow of kernels does not commute with measurable presentations of distributions of measure-valued mappings. In this paper we improve the approach suggested in \cite{LeJan_Raimond_1, LeJan_Raimond_2}. Starting from a consistent sequence of Feller transition functions  we prove the existence of a single idempotent measurable presentation $\mathfrak{p}$ of corresponding distributions of  kernels. Further, we construct a stochastic flow of kernels $(K_{s,t}:-\infty<s\leq t<\infty)$ in such a way that equalities $K_{s,t}(\omega)=\mathfrak{p}\left(K_{s,t}(\omega)\right)$ are satisfied without exceptions in $s,t,\omega$. Moreover, we achieve measurability of the  mapping $(s,t,\omega)\mapsto K_{s,t}(\omega).$  Together with equalities  $K_{s,t}(\omega)=\mathfrak{p}\left(K_{s,t}(\omega)\right)$ this implies measurability of the mapping $(s,t,\omega,x)\mapsto K_{s,t}(\omega, x)=\mathfrak{p}(K_{s,t}(\omega))(x).$

The paper is organized as follows. In Section \ref{sec:def} we give definitions of consistent sequences of  Feller transition functions, Feller convolution semigroups in the space of kernels and stochastic flows of kernels on a locally compact separable metric space $M$. Also, we show that a Feller convolution semigroup on $M$  defines a consistent  sequence of Feller transition functions on $M$ that determines finite-point motions with respect to the semigroup,  and a stochastic flow of kernels in $M$ defines a Feller convolution semigroup in the space of kernels on $M$ that defines distributions of kernels in a flow. In Section \ref{sec:main_thm_1} we prove that any consistent sequence of Feller transition functions on $M$ defines a unique Feller convolution semigroup in the space of kernels on $M$ with finite-point motions determined by the given sequence of transition functions. This result was obtained in \cite{LeJan_Raimond_1, LeJan_Raimond_2}. Our approach enables to construct a single idempotent measurable presentation $\mathfrak{p}$ of all distributions from a Feller convolution semigroup (Theorem \ref{thm:FCS_kernels}). In Section \ref{sec:main_thm_2} we prove that from any Feller convolution 
semigroup $(\nu_t:t\geq 0)$ in the space of kernels on $M$ one can construct a stochastic flow of kernels $(K_{s,t}:-\infty<s\leq t<\infty)$ in $M,$ for which the distribution of each  kernel $K_{s,t}$ coincides with $\nu_{t-s},$ the mapping $(s,t,\omega)\mapsto K_{s,t}(\omega)$ is measurable and equalities $\mathfrak{p}(K_{s,t}(\omega))=K_{s,t}(\omega)$  hold without exceptions in $(s,t,\omega)$ (Theorem \ref{thm:SFK}). Auxiliary Propositions \ref{prp:approximation} and \ref{prp:convergence} about approximations of stochastic flows of kernels seem to be new and interesting on their own. Another interesting consequence of our approach is that constructions of Feller convolution semigroups and of stochastic flows of kernels are done using approximating procedures that are very similar in their nature, but differ in the domain of approximation: the approximation is in space for Feller convolution semigroups and is in time for stochastic flow of kernels.

Finally, we note that our definitions of stochastic flows of kernels and Feller convolution semigroups are slightly different from the ones  given in \cite{LeJan_Raimond_1, LeJan_Raimond_2}. To show equivalence of definitions we give full proofs of several known statements from \cite{LeJan_Raimond_1, LeJan_Raimond_2}.

\section{Definitions, Preliminaries and Main Results}
\label{sec:def}

Let $(M,\rho)$ be a locally compact separable metric space equipped with the Borel $\sigma$-field $\mathcal{B}(M)$. Without loss of generality we assume that all $\rho$-bounded sets are relatively compact. In particular, $(M,\rho)$ is a complete separable  metric space. By $C(M)$ we denote the space of bounded continuous functions on $M$, and by $C_0(M)$ we denote the space of all continuous functions $f\in C(M)$ that vanish at infinity in the sense that for any $\varepsilon>0$ there exists a compact $C\subset M,$ such that $\sup_{x\in M\setminus C}|f(x)|\leq \varepsilon.$ With respect to the norm $\|f\|=\sup_M |f|$ the space $C_0(M)$ is a separable Banach space. 
$\mathcal{P}(M)$ denotes the space of all Borel probability measures on $M.$

Let $\hat{M}$ be the one-point compactification of $M.$ The following construction will be useful in our considerations. Write $M$ as a union $M=\bigcup^\infty_{j=1}L_j$ of compact sets $L_j,$ such that $L_j$ is contained in the interior of $L_{j+1}.$ For each $j$ fix a continuous function $\zeta_j:\hat{M}\to [0,1],$ such that $\zeta_j|_{L_j}=1$ and the support of $\zeta_j$ is contained in the interior of $L_{j+1}.$ Sequences $(L_j:j\geq 1)$ and $(\zeta_j: j\geq 1)$ will be called  exhaustive. 
   

The space $\mathcal{P}(\hat{M})$ equipped with  the topology of weak convergence is a compact metrizable space. Let $\hat{d}$ be the corresponding metric on $\mathcal{P}(\hat{M}).$ The set $\mathcal{P}(M)$  is a $G_\delta$ subset in  $\mathcal{P}(\hat{M}),$ hence is a Polish space \cite[Ch. II, Th. 6.5]{Parthasarathy}. Denote by $d$ a  metric on $\mathcal{P}(M)$ that is compatible with the topology of weak convergence and turns $\mathcal{P}(M)$ into a complete separable  metric space. 

\subsection{Consistent sequences of Feller transition functions}
\label{subsec:CSFTF}

For $1\leq k\leq n$ denote by $S_{k,n}$ the set of all injections $\sigma:\{1,\ldots,k\}\to \{1,\ldots,n\}.$ Any $\sigma\in S_{k,n}$ defines a mapping $\pi_\sigma:M^n\to M^k,$ $\pi_\sigma x=(x_{\sigma(1)},\ldots,x_{\sigma(k)}).$

Assume that for each $n\in \mathbb{N}$ a Feller transition function $\mathsf{P}^{(n)}$ on $M^n$ is defined.

\begin{dfn}\label{def:CSCFTF} \cite[Def. 1.1]{LeJan_Raimond_1} A sequence $(\mathsf{P}^{(n)}: n\in \mathbb{N})$ is called a  consistent sequence of Feller transition functions on $M,$  if

for all $1\leq k\leq n,$ $\sigma\in S_{k,n},$ $x\in M^n$ and $t\geq 0$
\begin{equation}
\label{eq:TF_consistency}
\mathsf{P}^{(n)}_t(x)\circ \pi^{-1}_{\sigma}=\mathsf{P}^{(k)}_t(\pi_\sigma x).
\end{equation}

\end{dfn}

The following Lemma contains one useful property of Feller transition functions.

\begin{lem}
\label{lem:Feller_property} Let $(\mathsf{P}_t: t\geq 0)$ be a Feller transition function on a locally compact separable metric space $M.$ Then for any compact $C\subset M,$ $T\geq 0$ and $\varepsilon>0,$ there exists compact $L\subset M,$ such that 
$$
\inf_{x\in C,t\in [0,T]}\mathsf{P}_t(x,L)\geq 1-\varepsilon.
$$
\end{lem}

\begin{proof} 
Feller property implies that the map $(t,x)\mapsto \mathsf{P}_t(x)\in \mathcal{P}(M)$ is continuous. In particular, the set $\{\mathsf{P}_t(x): t\in [0,T], x\in C\}$ is compact in $\mathcal{P}(M).$ The result follows from Prokhorov's theorem \cite[Th. 6.7, Ch. II]{Parthasarathy}.

\end{proof}

\subsection{Feller convolution semigroups in the space of kernels}

A kernel on $M$ is a measurable mapping $K:M\to \mathcal{P}(M).$ By $E$ we denote the set of all kernels on $M.$  For $K_1,K_2\in E$ denote by $K_1K_2$ a kernel 
$$
K_1K_2(x)=\int_M K_2(y)K_1(x,dy).
$$
For $\mu\in \mathcal{P}(M)$ we denote by $\mu K$ a probability measure $\mu K(B)=\int_M K(x,B)\mu(dx)$, and for a bounded measurable function $f:M\to \mathbb{R}$ we denote by $Kf$ a measurable function $Kf(x)=\int_M f(y) K(x,dy)$.

The set $E$ is equipped with the cylindrical  $\sigma$-field $\mathcal{E}$ --  
the smallest $\sigma$-field on $E$ with respect to which all mappings $ K\mapsto K(x),$ $x\in M,$ are $\mathcal{E}/\mathcal{B}(\mathcal{P}(M))$-measurable.

\begin{dfn}\label{def:regular_measure_kernels} \cite[Def. 1.2, Def. 2.1]{LeJan_Raimond_1} 
A probability measure $\nu$ on $(E,\mathcal{E})$ is called {\it regular}, if there exists a mapping $\mathfrak{p}:E\to E,$ such that the mapping $E\times M\ni (K,x)\mapsto \mathfrak{p}(K)(x)\in  \mathcal{P}(M)$ is measurable, and  for all $x\in M,$ $\mathfrak{p}(K)(x)=K(x)$ $\nu$-a.s.

\end{dfn}

The mapping $\mathfrak{p}$ is called a measurable presentation of a regular measure $\nu.$  Let $\nu_1,$ $\nu_2$ be regular probability measures on $(E,\mathcal{E}),$ and let $\mathfrak{p}$ be a measurable presentation of $\nu_2.$ Then the mapping $(K_1,K_2)\mapsto K_1 \mathfrak{p}(K_2)$ is $\mathcal{E}^{\otimes 2}/\mathcal{E}$-measurable and its distribution with respect to the product measure $\nu_1\otimes \nu_2$ is independent from the choice of $\mathfrak{p}.$ The latter distribution is denoted by $\nu_1*\nu_2$ and is called a  convolution of $\nu_1$ and $\nu_2$ \cite{LeJan_Raimond_1}.

\begin{dfn}\label{def:FCS_kernel} \cite[Def. 1.4, Def. 1.5]{LeJan_Raimond_1} 
A family $(\nu_t:t\geq 0)$ of regular probability measures on $(E,\mathcal{E})$ is called a  Feller convolution semigroup in the space of kernels on $M$, if 

\begin{enumerate}

\item for all $t,s\geq 0$, $\nu_t*\nu_s=\nu_{t+s};$

\item for any $f\in C_0(M)$ and any $\varepsilon>0,$
$$
\lim_{t\to 0+}\sup_{x\in M}\nu_t\left\{K: |Kf(x)-f(x)|\geq \varepsilon\right\}=0;
$$

\item for any $f\in C_0(M),$  $t\geq 0,$ $x\in M$ and $\varepsilon>0,$
$$
\lim_{y\to x}\nu_t\left\{K: |Kf(y)-Kf(x)|\geq \varepsilon\right\}=0, \ \lim_{y\to \infty}\nu_t\left\{K: |Kf(y)|\geq \varepsilon\right\}=0.
$$

\end{enumerate}

\end{dfn}

To each Feller convolution semigroup in the space of kernels  $(\nu_t:t\geq 0)$ one can associate a consistent sequence of Feller transition functions $(\mathsf{P}^{(n)}: n\geq 1)$ as follows: for all $n\geq 1,$ $x\in M^n,$ $B\in\mathcal{B}(M^n),$ $t\geq 0,$

\begin{equation}
\label{eq:FCS_to_CSCFTF_kernels}
\mathsf{P}^{(n)}_t(x,B)=\int_E \left(\otimes^n_{i=1}K(x_i)\right)(B)\nu_t(dK).
\end{equation}

\begin{prp}
\label{prop:FCS_to_CSFTF}  $(\mathsf{P}^{(n)}: n\geq 1 )$ is a consistent sequence of  Feller transition functions on $M.$
\end{prp}

\begin{proof} Let $\mathfrak{p}_t$ be a measurable presentation of $\nu_t.$ Measurability of $\mathsf{P}^{(n)}_t(x,B)$ in $x$ and the Chapman-Kolmogorov equation for $\mathsf{P}^{(n)}$  follow from the representation
$$
\mathsf{P}^{(n)}_t(x,B)=\int_E \left(\otimes^n_{i=1}\mathfrak{p}_t(K)(x_i)\right)(B)\nu_t(dK),
$$
and the convolution semigroup property of $\nu.$


We verify consistency. Let $\sigma\in S_{k,n}.$ Then 
$$
\begin{aligned}
\mathsf{P}^{(n)}_t\left(x,\pi_\sigma^{-1}(B) \right)&=\int_{E} \left(\otimes^n_{j=1}K(x_j) \right)\left(\pi_\sigma^{-1}(B)\right)\nu_t(dK)\\
&=\int_{E} \left(\otimes^k_{j=1}K(x_{\sigma(j)}) \right)(B)\nu_t(dK)=\mathsf{P}^{(k)}_t\left(\pi_\sigma x,B\right).
\end{aligned}
$$

It remains to verify the Feller property of $\mathsf{P}^{(n)}.$ By the Stone-Weierstrass theorem, it is enough to consider functions $f\in C_0(M^n)$ of the form $f(x)=\prod^n_{j=1} g_j(x_j),$ $g_j\in C_0(M).$ Then
$$
\begin{aligned}
|\mathsf{P}^{(n)}_tf(x)& -\mathsf{P}^{(n)}_tf(y)|=\left|\int_E \prod^n_{j=1}Kg_j (x_j) \nu_t(dK)-\int_E \prod^n_{j=1}Kg_j(y_j) \nu_t(dK)\right|\\
&\leq \sum^n_{k=1}\left|\int_{E}\left( \prod^{k}_{j=1}Kg_j(x_j)\prod^n_{j=k+1}Kg_j (y_j)- \prod^{k-1}_{j=1}Kg_j (x_j)\prod^n_{j=k}Kg_j(y_j)\right)\nu_t(dK)\right|\\
&\leq \sum^n_{k=1}\left|\int_{E} \prod^{k-1}_{j=1}Kg_j(x_j)\prod^n_{j=k+1}Kg_j (y_j) \times \left(Kg_k(x_k)-Kg_k(y_k)\right)\nu_t(dK)\right|\\
&\leq 2n\prod^n_{j=1}\|g_j\|\times \sup_{1\leq k\leq n}\nu_t\{K: \ |Kg_k(x_k)-Kg_k(y_k)|\geq \varepsilon\}+n\varepsilon \prod^n_{j=1}(\|g_j\|+1)\\
&\to n\varepsilon \prod^n_{j=1}(\|g_j\|+1), \ y\to x.
\end{aligned}
$$
Since $\varepsilon>0$ is arbitrary, we deduce that $\mathsf{P}^{(n)}_t f$ is continuous on $M^n.$

For any $\varepsilon>0$ there exists a compact $L\subset M,$ such that 
$$
\sup_{1\leq k\leq n}\sup_{y\not\in L}\nu_t\{K: |Kg_k(y)|\geq \varepsilon\}\leq \varepsilon.
$$
If $x\not\in L^n$ with, say, $x_k\not\in L,$  then

$$
\begin{aligned}
|\mathsf{P}^{(n)}_tf(x)|&=\left|\int_E \prod^n_{j=1}Kg_j (x_j) \nu_t(dK)\right|\\
&\leq \prod^n_{j=1}\|g_j\|\times \nu_t\{K: |Kg_k(x_k)|\geq \varepsilon\} +\varepsilon \prod^n_{j=1}(\|g_j\|+1)\leq 2\varepsilon \prod^n_{j=1}(\|g_j\|+1).
\end{aligned}
$$
It follows that $\lim_{x\to\infty}\mathsf{P}^{(n)}_tf(x)=0.$ So, $\mathsf{P}^{(n)}_t(C_0(M^n))\subset C_0(M^n).$

Further,
$$
\begin{aligned}
|\mathsf{P}^{(n)}_tf(x)&-f(x)|=\left|\int_E \left(\prod^n_{j=1}Kg_j(x_j)-\prod^n_{j=1} g_j(x_j)\right)\nu_t(dK)\right|\\
&\leq \sum^n_{k=1}\left|\int_{E}\left( \prod^{k}_{j=1}Kg_j(x_j)\prod^n_{j=k+1}g_j(x_j)- \prod^{k-1}_{j=1}Kg_j(x_j)\prod^n_{j=k}g_j(x_j)\right)\nu_t(dK)\right|\\
&\leq \sum^n_{k=1}\left|\int_{E} \prod^{k-1}_{j=1}Kg_j(x_j)\prod^n_{j=k+1}g_j(x_j)\times \left(Kg_k(x_k)-g_k(x_k)\right)\nu_t(dK)\right|\\
&\leq 2n\prod^n_{j=1}\|g_j\|\times  \sup_{1\leq k\leq n}\sup_{y\in M}\nu_t\{K: \ |Kg_k(y)-g_k(y)|\geq \varepsilon\}+n\varepsilon \prod^n_{j=1}(\|g_j\|+1).
\end{aligned}
$$
It follows that 
$$
\begin{aligned}
\sup_{x\in M^n}|\mathsf{P}^{(n)}_tf(x)-f(x)|&\leq 2n\prod^n_{j=1}\|g_j\|\times  \sup_{1\leq k\leq n}\sup_{y\in M}\nu_t\{K: \ |Kg_k(y)-g_k(y)|\geq \varepsilon\}\\
& \ \ \ +n\varepsilon \prod^n_{j=1}(\|g_j\|+1)\to n\varepsilon \prod^n_{j=1}(\|g_j\|+1), \ t\to 0.
\end{aligned}
$$
Since $\varepsilon>0$ is arbitrary, we deduce that $(\mathsf{P}^{(n)}_t: t\geq 0)$ is strongly continuous at $t=0.$
\end{proof}

The sequence $(\mathsf{P}^{(n)}: n\geq 1)$ completely determines the semigroup $(\nu_t:t\geq 0).$ To show this we introduce an algebra $\mathbb{A}_n(M)$   of continuous functions on $\mathcal{P}(M)^n,$ that consists of all functions $g:\mathcal{P}(M)^n\to \mathbb{R}$ of the form
\begin{equation}
\label{eq:functions_algebra}
g(\mu_1,\ldots,\mu_n)=\int_{M^N}f(y)\left(\mu_{i_1}\otimes \ldots\otimes \mu_{i_N}\right)(dy),
\end{equation}
where $f\in C_0(M^N),$ $(i_1,\ldots,i_N)\in \{1,\ldots,n\}^N.$

\begin{lem}
\label{lem:distribution_on_P(M)^n} A probability measure $\Pi$ on $\mathcal{P}(M)^n$ is completely determined by integrals of the form 
\begin{equation}
\label{eq:integrals_algebra}
\int_{\mathcal{P}(M)^n} g(\mu_1,\ldots,\mu_n)\Pi(d\mu),
\end{equation}
where $g\in \mathbb{A}_n(M).$
\end{lem}

\begin{proof} Let $M$ be compact. Then $\mathcal{P}(M)^n$ is a compact metric space and $\mathbb{A}_n(M)$ is dense in $C(\mathcal{P}(M)^n)$ by the Stone-Weierstrass theorem. Hence, integrals of the form \eqref{eq:integrals_algebra} with $g\in \mathbb{A}_n(M)$ define integrals of the form \eqref{eq:integrals_algebra} with $g\in C(\mathcal{P}(M)^n).$ In this case the result is proved.

In general case, consider the one-point compactification $\hat{M}$ of $M.$  $\Pi$ can be viewed as a probability measure on $\mathcal{P}(\hat{M})^n.$ It is completely determined by integrals of the form \eqref{eq:integrals_algebra} with $g\in \mathbb{A}_n(\hat{M}).$ Consider $g\in \mathbb{A}_n(\hat{M})$ of the form 
$$
g(\mu_1,\ldots,\mu_n)=\int_{\hat{M}^N}f(y)(\mu_{i_1}\otimes \ldots\otimes \mu_{i_N})(dy),
$$
where $f\in C(\hat{M}^N),$ $(i_1,\ldots,i_N)\in \{1,\ldots,n\}^N.$ Let $g_j\in \mathbb{A}_n(M)$ be defined as
$$
g_j(\mu_1,\ldots,\mu_n)=\int_{M^N}f(y)\zeta_j^{\otimes N}(y)(\mu_{i_1}\otimes \ldots\otimes \mu_{i_N})(dy),
$$
where $(\zeta_j:j\geq 1)$ is the exhaustive sequence introduced in the beginning of Section \ref{sec:def}. The result follows, since 
$$
\begin{aligned}
\int_{\mathcal{P}(M)^n}g(\mu_1,\ldots,\mu_n)\Pi(d\mu)&=\int_{\mathcal{P}(M)^n}\left(\int_{M^N}f(y)(\mu_{i_1}\otimes \ldots\otimes \mu_{i_N})(dy)\right)\Pi(d\mu)\\
&=\lim_{j\to\infty}\int_{\mathcal{P}(M)^n}g_j(\mu_1,\ldots,\mu_n)\Pi(d\mu).
\end{aligned}
$$

\end{proof}

\begin{lem}
\label{CSFTF_determines_FCS} The sequence $(\mathsf{P}^{(n)}: n\geq 1)$ completely determines the Feller convolution semigroup $(\nu_t:t\geq 0).$
\end{lem}

\begin{proof}

The probability measure $\nu_t$ is completely determined by distributions of $\mathcal{P}(M)^n$-valued random elements $(K(x_1),\ldots,K(x_n)),$ where $x\in M^n,$ $n\geq 1.$ Hence, $\nu_t$ is completely determined by integrals of the form 
\begin{equation}
\label{eq:nu_t_determination}
\int_{E}\left(\int_{M^N}f(y)\left(K(x_{i_1})\otimes \ldots\otimes K(x_{i_N})\right)(dy)\right)\nu_t(dK),
\end{equation}
where $f\in C_0(M^N),$ $(i_1,\ldots,i_N)\in \{1,\ldots,n\}^N.$ It remains to note that \eqref{eq:nu_t_determination}  is  equal to $\mathsf{P}^{(N)}_tf(x_{i_1},\ldots,x_{i_N}).$

\end{proof}

In  \cite{LeJan_Raimond_1} it was proved that to any consistent sequence of Feller transition functions $(\mathsf{P}^{(n)}: n\geq 1)$ on $M$ there corresponds a unique Feller convolution semigroup $(\nu_t: t\geq 0)$ on $E,$ such that  \eqref{eq:FCS_to_CSCFTF_kernels} holds. Theorem \ref{thm:FCS_kernels} gives a strengthed version of this result. The main difference is that we find one idempotent measurable presentation $\mathfrak{p}$ of all measures $\nu_t.$

\begin{thm}
\label{thm:FCS_kernels}
Let $(\mathsf{P}^{(n)}: n\geq 1)$ be a consistent sequence of Feller transition functions on $M$.  There exists a unique Feller convolution semigroup  $(\nu_t:t\geq 0)$ that satisfies \eqref{eq:FCS_to_CSCFTF_kernels}. Moreover,  there exists a mapping  $\mathfrak{p}:E\to E$ which is a measurable presentation of every measure $\nu_t,$ $t\geq 0,$ and satisfies the relation $\mathfrak{p}\circ \mathfrak{p}=\mathfrak{p}.$ 
\end{thm}

\subsection{Stochastic flows of kernels}
\label{subsec:SFK}

\begin{dfn}\label{def:SFK}\cite[Def. 2.3]{LeJan_Raimond_1}, \cite[Def. 7]{LeJan_Raimond_2} A stochastic flow of kernels in $M$ is a family $K=(K_{s,t}:-\infty<s\leq t<\infty)$  of random elements in $(E,\mathcal{E})$ that are defined on a common  probability space $(\Omega,\mathcal{A},\mathsf{P})$  and satisfy the following properties: 

\begin{enumerate}
\item the law of $K_{s,t}$ is regular and coincides with the law of $K_{0,t-s};$ 

\item for all $r\leq s\leq t,$ $x\in M$ and any measurable presentation $\mathfrak{p}_{t-s}$ of the law of $K_{s,t},$ 
$$
K_{r,t}(x)=K_{r,s}\mathfrak{p}_{t-s}(K_{s,t})(x) \ \  \mathsf{P}-\mbox{a.s.};
$$

\item if $t_1\leq \ldots\leq t_n,$ then $K_{t_1,t_2},\ldots,K_{t_{n-1},t_n}$ are mutually independent;

\item for any $f\in C_0(M)$  and $\varepsilon>0,$
$$
\lim_{t\to 0+}\sup_{x\in M}\mathsf{P}\left\{|K_{0,t}f(x)-f(x)|\geq \varepsilon\right\}=0;
$$

\item for any $f\in C_0(M),$ $x\in M,$ $t\geq 0$ and $\varepsilon>0,$
$$
\lim_{y\to x}\mathsf{P}\{|K_{0,t}f(y)-K_{0,t}f(x)|\geq \varepsilon\}=0, \ \lim_{y\to \infty}\mathsf{P}\{|K_{0,t}f(y)|\geq \varepsilon\}=0.
$$

\end{enumerate}

\end{dfn}

Let $\nu_t$ denote the law of $K_{0,t}.$ Clearly, $(\nu_t: t\geq 0)$ is a Feller convolution semigroup in the space   of kernels on $M$. The converse result is also true: if $(\nu_t: t\geq 0)$ is a Feller convolution semigroup in the space  of kernels on $M,$ then there exists  a stochastic flow of kernels $K=(K_{s,t}:-\infty<s\leq t<\infty)$ in $M,$  such that for all $s\leq t$ the law of $K_{s,t}$ coincides with $\nu_{t-s}$ \cite[Th 2.1]{LeJan_Raimond_1}.  We prove that such stochastic flow can be always constructed as a measurable function from $(s,t,\omega)$ that satisfies relations $\mathfrak{p}_{t-s}(K_{s,t}(\omega))=K_{s,t}(\omega)$ without exceptions in $(s,t,\omega).$

\begin{thm}
\label{thm:SFK}
Let $(\nu_t:t\geq 0)$ be a Feller convolution semigroup in the space of kernels on $M$ and let $\mathfrak{p}$ be a common idempotent measurable representation of measures $\nu_t$ (Theorem \ref{thm:FCS_kernels}). There exists a stochastic flow of kernels $K$ in $M,$ such that 

\begin{enumerate}

\item For all $s\leq t$ the law of $K_{s,t}$ coincides with $\nu_{t-s};$

\item The mapping $(s,t,\omega)\mapsto K_{s,t}(\omega)$ is jointly measurable;

\item $K_{s,s}(\omega)(x)=\delta_x$ for all $s\in \mathbb{R},$ $x\in M,$ $\omega\in \Omega;$

\item $\mathfrak{p}(K_{s,t}(\omega))=K_{s,t}(\omega)$ for all $s\leq t$ and $\omega\in \Omega.$ 

\end{enumerate}

\end{thm}

\section{Proof of the Theorem \ref{thm:FCS_kernels}}
\label{sec:main_thm_1}

\subsection{Probability measures $\Pi^{(n)}_t(x)$.}

Out of the sequence $(\mathsf{P}^{(n)}: n\geq 1)$ we construct for any $n\geq 1,$ $x\in M^n$ and $t\geq 0$ a probability measure $\Pi^{(n)}_t(x)$ on $\mathcal{P}(M)^n$ which will be the distribution of $K\mapsto (K(x_1),\ldots,K(x_n))$ under $\nu_t.$

Recall the dense algebra $\mathbb{A}_n(\hat{M})$ in the space of continuous functions on $\mathcal{P}(\hat{M})^n.$ Let $g\in \mathbb{A}_n(\hat{M})$ be of the form \eqref{eq:functions_algebra} with $f\in C(\hat{M}^N),$ $(i_1,\ldots,i_N)\in \{1,\ldots,n\}^N.$  For $x\in M^n$ and $t\geq 0$ define 
$$
\Pi^{(n)}_t(x) g=\int_{M^N} f(y) \mathsf{P}^{(N)}_t((x_{i_1},\ldots,x_{i_N}),dy).
$$

\begin{lem}
\label{lem:Pi_correctness} $\Pi^{(n)}_t(x)$ is a correctly defined linear non-negative functional on $\mathbb{A}_n(\hat{M}),$ such that $\Pi^{(n)}_t(x)1=1.$
\end{lem}

\begin{proof} Let us check correctness of the definition of $\Pi^{(n)}_t(x).$ Assume that $g\in \mathbb{A}_n(\hat{M})$ has two representations: for all $(\mu_1,\ldots,\mu_n)\in \mathcal{P}(\hat{M})^n$

$$
\begin{aligned}
g(\mu_1,\ldots,\mu_n)&=\int_{\hat{M}^N} f(y)\left(\mu_{i_1}\otimes \ldots\otimes \mu_{i_N}\right)(dy)\\
&=\int_{\hat{M}^R} v(y)\left(\mu_{j_1}\otimes \ldots\otimes \mu_{j_R}\right)(dy),
\end{aligned}
$$
where $f\in C(\hat{M}^N),$ $(i_1,\ldots,i_N)\in \{1,\ldots,n\}^N,$ $v\in C(\hat{M}^R),$ $(j_1,\ldots,j_R)\in \{1,\ldots,n\}^R.$ Consider injections $\sigma:\{1,\ldots,N\}\to \{1,\ldots,N+R\},$ $\delta:\{1,\ldots,R\}\to \{1,\ldots,N+R\},$ defined by
$$
\sigma(i)=i, \ 1\leq i\leq N, \ \delta(j)=N+j, \ 1\leq j\leq R.
$$
Then 
$$
\begin{aligned}
g(\mu_1,\ldots,\mu_n)&=\int_{\hat{M}^{N+R}} f\circ \pi_\sigma(y)\left(\mu_{i_1} \otimes \ldots\otimes \mu_{i_N}\otimes \mu_{j_1}\otimes \ldots\otimes \mu_{j_R}\right)(dy)\\
&=\int_{\hat{M}^{N+R}} v\circ \pi_\delta(y)\left(\mu_{i_1} \otimes \ldots\otimes \mu_{i_N}\otimes \mu_{j_1}\otimes \ldots\otimes \mu_{j_R}\right)(dy).
\end{aligned}
$$
By consistency,
$$
\int_{M^{N+R}}f\circ \pi_\sigma(y)\mathsf{P}^{(N+R)}_t((x_{i_1},\ldots,x_{i_N},x_{j_1},\ldots,x_{j_R}),dy)=\int_{M^N}f(y)\mathsf{P}^{(N)}_t((x_{i_1},\ldots,x_{i_N}),dy),
$$
$$
\int_{M^L}v\circ \pi_\delta(y)\mathsf{P}^{(N+R)}_t((x_{i_1},\ldots,x_{i_N},x_{j_1},\ldots,x_{j_R}),dy)=\int_{M^R}v(y)\mathsf{P}^{(R)}_t((x_{j_1},\ldots,x_{j_R}),dy).
$$
So, it is enough to consider the case $(i_1,\ldots,i_N)=(j_1,\ldots,j_R).$ Further, it is enough to prove that equality
\begin{equation}
\label{eq:integral_f_zero}
\int_{\hat{M}^N}f(y) \left(\mu_{i_1}\otimes \ldots\otimes \mu_{i_N}\right)(dy)=0, \ (\mu_1,\ldots,\mu_n)\in \mathcal{P}(\hat{M})^n,
\end{equation}
implies 
$$
\int_{M^N}f(y)\mathsf{P}^{(N)}_t((x_{i_1},\ldots,x_{i_N}),dy)=0.
$$ 
Assume that \eqref{eq:integral_f_zero} holds. For $s\in \{1,\ldots,n\}$ denote 
$$
I_s=\{k\in \{1,\ldots,N\}: \ i_k=s\}
$$
and let $m_s$ be the number of elements in $I_s$. Denote by $S_{N,N}(I_1,\ldots,I_n)$ the set of all permutations $\sigma\in S_{N,N}$ such that  $\sigma(I_s)=I_s$ for all $s\in \{1,\ldots,n\}.$ 
Let 
$$
\tilde{f}(y)=\frac{1}{m_1! \ldots  m_n!}\sum_{\sigma\in S_{N,N}(I_1,\ldots,I_n)}f\circ \pi_\sigma(y).
$$
We note that 
$$
\int_{\hat{M}^N}\tilde{f}(y)\left(\mu_{i_1}\otimes \ldots\otimes \mu_{i_N}\right)(dy)=0.
$$
By consistency,
$$
\begin{aligned}
\int_{M^N} & \tilde{f}(y)\mathsf{P}^{(N)}_t((x_{i_1},\ldots,x_{i_N}),dy)\\
&=\frac{1}{m_1!\ldots m_n!}\sum_{\sigma\in S_{N,N}(I_1,\ldots,I_n)}\int_{M^N} f\circ \pi_\sigma(y)\mathsf{P}^{(N)}_t((x_{i_1},\ldots,x_{i_N}),dy)\\
&=\int_{M^N}f(y) \mathsf{P}^{(N)}_t((x_{i_1},\ldots,x_{i_N}),dy).
\end{aligned}
$$
So, we may assume that $f\circ \pi_\sigma=f$ for all $\sigma\in S_{N,N}(I_1,\ldots,I_n).$ We will show that equality \eqref{eq:integral_f_zero} implies $f(z)=0$ for all $z\in \hat{M}^N.$ By Fubini's theorem it is enough to consider the case $n=1.$ In this case $f\in C(\hat{M}^N)$ is symmetric and 
$$
\int_{\hat{M}^N}f(y)\mu^{\otimes N}(dy)=0
$$ for all finite measures $\mu$ on $\hat{M}.$ Let $z\in \hat{M}^N.$ Then 
$$
\int_{\hat{M}^N} f(y)\left(p_1\delta_{z_1}+\ldots+p_N\delta_{z_N}\right)^{\otimes N}(dy)=0
$$
for all $p_1,\ldots,p_N>0.$ Expanding and using symmetry of $f$, we get 
$$
\sum_{k_1+\ldots+k_N=N}\frac{N!}{k_1!\ldots k_N!}p^{k_1}_1\ldots p^{k_N}_N f(\underbrace{z_1,\ldots,z_1}_{k_1},\ldots,\underbrace{z_N,\ldots,z_N}_{k_N})=0.
$$
Differentiating in $p_1,\ldots,p_N$ at $p_1=\ldots=p_N=0$ we find that $f(z)=0.$ Correctness of the definition of $\Pi^{(n)}_t(x)$ is verified.  Independence of $\Pi^{(n)}_t(x)g$ from the representation of $g$ in the form \eqref{eq:functions_algebra} implies linearity of $\Pi^{(n)}_t(x).$

It remains to verify that the linear functional  $\Pi^{(n)}_t(x):\mathbb{A}_n(\hat{M})\to \mathbb{R}$ is non-negative. Assume that for all $(\mu_1,\ldots,\mu_n)\in \mathcal{P}(\hat{M}^n)$
$$
g(\mu_1,\ldots,\mu_n)=\int_{\hat{M}^N}f(y)\left(\mu_{i_1}\otimes \ldots \otimes \mu_{i_N}\right)(dy)\geq 0.
$$
As before, denote $I_s=\{k\in \{1,\ldots,N\}: \ i_k=s\},$ $s\in \{1,\ldots,n\},$
and let $m_s$ be the number of elements in $I_s$.
For an integer $L$ denote 
$$
x^{(L)}=(\underbrace{x_1,\ldots,x_1}_{L},\underbrace{x_2,\ldots,x_2}_{L},\ldots,\underbrace{x_n,\ldots,x_n}_{L}).
$$
We have
$$
\int_{M^{Ln}}g\left(\frac{1}{L}\sum^L_{j=1}\delta_{y_{j}},\frac{1}{L}\sum^L_{j=1}\delta_{y_{L+j}},\ldots\frac{1}{L}\sum^L_{j=1}\delta_{y_{(n-1)L+j}}\right)\mathsf{P}^{(Ln)}_t(x^{(L)},dy)\geq 0.
$$
Hence,
\begin{equation}
\label{eq:non-negativity}
\frac{1}{L^N} \sum^L_{j_1,\ldots,j_N=1}
\int_{M^{Ln}}f\bigg(y_{(i_1-1)L+j_1},\ldots,y_{(i_N-1)L+j_N}\bigg)\mathsf{P}^{(Ln)}_t(x^{(L)},dy)
\geq 0.
\end{equation}

Assume that for every $s\in \{1,\ldots,n\}$ all $j_k$ with $k\in I_s$ are distinct. By consistency,
$$
\begin{aligned}
\int_{M^{Ln}}&f\bigg(y_{(i_1-1)L+j_1},\ldots,y_{(i_N-1)L+j_N}\bigg)\mathsf{P}^{(Ln)}_t(x^{(L)},dy)\\
&=\int_{M^N}f(y)\mathsf{P}^{(N)}_t((x_{i_1},\ldots,x_{i_N}),dy)=\Pi^{(n)}_t(x)g.
\end{aligned}
$$
So, \eqref{eq:non-negativity} implies 
$$
\frac{\prod^n_{i=1}L(L-1)\ldots (L-m_i+1)}{L^N}\Pi^{(n)}_t(x)g+R_L\geq 0,
$$
where 
$$
|R_L|\leq \left(1-\frac{\prod^n_{i=1}L(L-1)\ldots (L-m_i+1)}{L^N}\right)\|f\|.
$$
Taking the limit $L\to\infty,$ we obtain $\Pi^{(n)}_t(x)g\geq 0.$

\end{proof}

Lemma \ref{lem:Pi_correctness} implies that for every $n\geq 1,$ $x\in M^n$ and $t\geq 0$ the linear functional $\Pi^{(n)}_t(x)$ is represented by a probability measure on $\mathcal{P}(\hat{M})^n.$ This measure will be also denoted by $\Pi^{(n)}_t(x).$ In particular, the equality 
$$
\begin{aligned}
\int_{\mathcal{P}(\hat{M})^n}&\left(\int_{\hat{M}^N}f(y)\left(\mu_{i_1}\otimes \ldots \otimes \mu_{i_N}\right)(dy)\right)\Pi^{(n)}_t(x,d\mu)\\
&=\int_{M^N}f(y)\mathsf{P}^{(N)}_t((x_{i_1},\ldots,x_{i_N}),dy)
\end{aligned}
$$
holds for all $f\in C(\hat{M}^N),$ $(i_1,\ldots,i_N)\in \{1,\ldots,n\}^N.$ Next lemmata contain  some useful  properties of measures $\Pi^{(n)}_t(x).$

\begin{lem}
\label{lem:properties_of_Pi_1} \begin{enumerate}

\item For all $\sigma\in S_{k,n},$ $1\leq k\leq n,$ and all  $x\in M^n,$  $t\geq 0,$ 
$$
\Pi^{(n)}_t(x)\circ \pi^{-1}_{\sigma}=\Pi^{(k)}_t(\pi_{\sigma}x).
$$

\item For all $x\in M,$ $t\geq 0,$
$$
\Pi^{(2)}_t((x,x),\Delta)= 1,
$$
where $\Delta=\{(\mu,\mu): \ \mu\in\mathcal{P}(\hat{M})\}.$

\item For all  $n\geq 1,$ $x\in M^n,$ $t\geq 0,$
$$
\Pi^{(n)}_t(x,\mathcal{P}(M)^n)=1.
$$

\item For any  $g\in C(\mathcal{P}(\hat{M})^n)$ the mapping $(t,x)\mapsto \Pi^{(n)}_t(x)g$ is continuous.

\end{enumerate}
\end{lem}

\begin{proof} \begin{enumerate}

\item Let $\sigma\in S_{k,n}.$ Consider $g\in \mathbb{A}_k(\hat{M})$ of the form 
$$
g(\mu_1,\ldots,\mu_k)=\int_{\hat{M}^N}f(y)\left(\mu_{i_1}\otimes \ldots \otimes \mu_{i_N}\right)(dy),
$$
where $f\in C(\hat{M}^N),$ $(i_1,\ldots,i_N)\in \{1,\ldots,k\}^N.$ Then 
$$
g\circ \pi_\sigma(\mu_1,\ldots,\mu_n)=g(\mu_{\sigma(1)},\ldots,\mu_{\sigma(k)})=\int_{\hat{M}^N}f(y)\left(\mu_{\sigma(i_1)}\otimes \ldots \otimes \mu_{\sigma(i_N)}\right)(dy).
$$
So,
$$
\begin{aligned}
\int_{\mathcal{P}(\hat{M})^n}g\circ \pi_\sigma(\mu_1,\ldots,\mu_n) \Pi^{(n)}_t(x,d\mu)&=\int_{M^N}f(y)\mathsf{P}^{(N)}_t((x_{\sigma(i_1)},\ldots,x_{\sigma(i_N)}),dy)\\
&=\int_{\mathcal{P}(\hat{M})^k}g(\mu_1,\ldots,\mu_k) \Pi^{(k)}_t(\pi_\sigma x,d\mu).
\end{aligned}
$$
Equality $\Pi^{(n)}_t(x)\circ \pi^{-1}_{\sigma}=\Pi^{(k)}_t(\pi_\sigma x)$ is verified.

\item Let $g\in \mathbb{A}_1(\hat{M}),$ $g(\mu)=\int_{\hat{M}^N}f(y)\mu^{\otimes N}(dy).$ Then 
$$
\Pi^{(2)}_t g^{\otimes 2}((x,x))=\int_{M^{2N}}f^{\otimes 2}(y)\mathsf{P}^{(2N)}_t (\underbrace{(x,\ldots,x)}_{2N},dy)=\Pi^{(1)}_t g^2(x).
$$
By continuity, for all $g_1,g_2\in C(\mathcal{P}(\hat{M}))$ we have 
$$
\int_{\mathcal{P}(\hat{M})^2}g_1(\mu_1)g_2(\mu_2)\Pi^{(2)}_t ((x,x),d\mu)=\int_{\mathcal{P}(\hat{M})}g_1(\mu)g_2(\mu)\Pi^{(1)}_t (x,d\mu).
$$
Hence, for any closed sets $F_1,F_2\subset \mathcal{P}(\hat{M}),$ 
$$
\Pi^{(2)}_t((x,x), F_1\times F_2)=\Pi^{(1)}_t(x, F_1\cap F_2). 
$$
It follows that $\Pi^{(2)}_t((x,x),\Delta)=1.$

\item Let $x\in M^n$ and $g_k(\mu_1,\ldots,\mu_n)=\prod^n_{i=1}\int_{\hat{M}}\zeta_k(y)\mu_i(dy).$ Then
$$
\begin{aligned}
1&\geq \int_{\mathcal{P}(\hat{M})^n}\prod^n_{i=1}\mu_i(M)\Pi^{(n)}_t(x,d\mu)\geq \int_{\mathcal{P}(\hat{M})^n}g_k(\mu_1,\ldots,\mu_n)\Pi^{(n)}_t(x,d\mu)\\
&=\int_{M^n}\prod^n_{i=1} \zeta_k(y_i) \mathsf{P}^{(n)}_t(x,dy)\geq \mathsf{P}^{(n)}_t(x,L^n_k).
\end{aligned}
$$
Taking the limit $k\to\infty$ we deduce that $\int_{\mathcal{P}(\hat{M})^n}\prod^n_{i=1}\mu_i(M)\Pi^{(n)}_t(x,d\mu)=1$ and $\mu_1(M)=\ldots=\mu_n(M)=1$ for $\Pi^{(n)}_t(x)$-a.a. $(\mu_1,\ldots,\mu_n)\in\mathcal{P}(\hat{M})^n.$

\item Let $g\in \mathbb{A}_n(\hat{M})$ be of the form 
$$
g(\mu_1,\ldots,\mu_n)=\int_{\hat{M}^N}f(y)\left(\mu_{i_1}\otimes \ldots\otimes \mu_{i_N}\right)(dy),
$$
where $f\in C(\hat{M}^N),$ $(i_1,\ldots,i_N)\in \{1,\ldots,n\}^N.$ Then
$$
\Pi^{(n)}_t(x)g=\int_{M^N}f(y)\mathsf{P}^{(N)}_t((x_{i_1},\ldots,x_{i_N}),dy).
$$
Fix $l\geq 1$ and $T\geq 0.$ By Feller property of $\mathsf{P}^{(N)}$ (Lemma \ref{lem:Feller_property}) for each $\varepsilon>0$ there exists $j\geq 1$ such that 
$$
\inf_{t\in [0,T], z\in L^N_l}\mathsf{P}^{(N)}_t(z,L^N_j)\geq 1-\varepsilon.
$$
For the function 
$$
g_j(\mu_1,\ldots,\mu_n)=\int_{\hat{M}^N}f(y)\zeta_j^{\otimes N}(y)\left(\mu_{i_1}\otimes \ldots\otimes \mu_{i_N}\right)(dy)
$$
we have an estimate
$$
\begin{aligned}
\ & \sup_{t\in [0,T], x\in L^n_l}\left|\Pi^{(n)}_t(x)g-\Pi^{(n)}_t(x)g_j\right|\\
=&\sup_{t\in [0,T]x\in L^n_l}\left|\int_{M^N\setminus L^N_j}f(y)(1-\zeta_j^{\otimes N}(y))\mathsf{P}^{(N)}_t((x_{i_1},\ldots,x_{i_N}),dy)\right|\leq \|f\|\varepsilon.
\end{aligned}
$$
On the other hand, equality
$$
\Pi^{(n)}_t(x)g_j=\mathsf{P}^{(N)}_t\left(f  \zeta^{\otimes N}_j\right)(x_{i_1},\ldots,x_{i_N})
$$
implies that the function $(t,x)\mapsto \Pi^{(n)}_t(x)g_j$ is continuous. Since $\varepsilon>0$ is arbitrary, we deduce that  the function $(t,x)\mapsto \Pi^{(n)}_t(x)g$ is continuous on $[0,T]\times L^n_l$ and thus on $[0,\infty)\times M^n.$

\end{enumerate}

\end{proof}

Denote $\Delta^c_{\varepsilon}=\{(\mu_1,\mu_2)\in \mathcal{P}(M)^2: d(\mu_1,\mu_2)\geq \varepsilon\}.$

\begin{lem}
\label{lem:properties_of_Pi_2} For any compact $C\subset M,$  $T\geq 0$ and $\varepsilon>0$
$$
\lim_{r\to 0+}\sup_{\substack{t\in [0,T], (x,y)\in C^2 \\  \rho(x,y)\leq r}}\Pi^{(2)}_t((x,y),\Delta^c_\varepsilon)=0.
$$
\end{lem}

\begin{proof} Assume the result does not hold. Then there is $\alpha>0$ and a sequence $(x_k,y_k,t_k)\in C^2\times [0,T],$ such that $\lim_{k\to\infty}\rho(x_k,y_k)=0$ and 
$$
\Pi^{(2)}_{t_k}((x_k,y_k),\Delta^c_\varepsilon)\geq \alpha.
$$
We may and do assume that $\lim_{k\to\infty}x_k=\lim_{k\to\infty}y_k=x\in C,$ and $\lim_{k\to\infty}t_k=t\in [0,T].$ Property (4) of the Lemma \ref{lem:properties_of_Pi_1} implies that $\Pi^{(2)}_{t_k}((x_k,y_k))\to \Pi^{(2)}_t((x,x))$ weakly as probability measures on $\mathcal{P}(\hat{M})^2,$ and as probability measures on $\mathcal{P}(M)^2.$  The Portmanteau theorem implies  
$$
\alpha\leq \limsup_{k\to\infty}\Pi^{(2)}_{t_k}((x_k,y_k),\Delta^c_\varepsilon)\leq \Pi^{(2)}_{t}((x,x),\Delta^c_\varepsilon)=0,
$$
since  $\Pi^{(2)}_t((x,x))$ is concentrated on $\Delta$ (property (2) of Lemma \ref{lem:properties_of_Pi_1}). Obtained contradiction proves the result.

\end{proof}

\subsection{Approximating procedure.}\label{subsec:approximating_procedure} The measure $\nu_t$ can be viewed as the distribution of a measure-valued process  $(K_{0,t}(x): x\in M).$ Let $Z$ be an at most countable dense set in $M.$ The idea of the construction is to define properly joint distribution of $(K_{0,t}(z): z\in Z)$ and recover the measure $\nu_t$ by certain limit procedure. To do this we fix a measurable mapping $\ell :\mathcal{P}(M)^\mathbb{N}\to \mathcal{P}(M)$ with the following property: for any relatively compact sequence $\mu=(\mu_n:n\in\mathbb{N}),$ $\ell(\mu)$ is a limit point of $\mu$ (see 
\cite[L. 7.1]{Raimond_Riabov} for the existence of such mapping).

Recall exhaustive sequence  $(L_j:j\geq 1)$ defined in the beginning of Section \ref{sec:def}. Lemma \ref{lem:properties_of_Pi_2} implies that there exists a sequence of positive numbers $(\varepsilon_j:j\geq 1)$ that is strictly decreasing to zero and is such that 
$$
(t,x,y)\in [0,j]\times L^2_j, \  \rho(x,y)\leq \varepsilon_j\Rightarrow \Pi^{(2)}_t((x,y),\Delta^c_{2^{-j}})\leq 2^{-j}.
$$

Let $m\mapsto z_m$ be a bijection between a subset $I$ of $\mathbb{N}$ and the set  $Z.$ For any $x\in M$ and any $j\geq 1$ we define 
\begin{equation}
\label{eq:approximating_sequence}
m^{x}_j=\inf\{m\in I:\rho(x,z_m)< \varepsilon_j/2\}.
\end{equation}

Note that $(m^x_j:j\geq 1)$ is a sequence in $I,$ and each mapping $x\mapsto m^x_j$  is measurable. 

Define mappings $i:\mathcal{P}(M)^I\to E,$ $e: E\to \mathcal{P}(M)^I,$ $\mathfrak{p}:E\to E$ as follows:  
$$
i(\mu)(x)=\ell\left(\left(\mu_{m^x_j}:j\geq 1\right)\right), \ e(K)=(K(z_m): m\in I), \ \mathfrak{p}=i\circ e.
$$

\begin{lem} \label{lem:measurable_presentation} Mappings $(x,\mu)\mapsto i(\mu)(x),$ $K\mapsto e(K),$ $(K,x)\mapsto \mathfrak{p}(K)(x)$ are measurable. Composition $e\circ i$ is the identity mapping on $\mathcal{P}(M)^I.$ Mapping $\mathfrak{p}$ satisfies the property $\mathfrak{p}\circ \mathfrak{p}=\mathfrak{p}.$ 
\end{lem}

\begin{proof}  By definition, $i(\mu)(x)=\ell\left(\left(\mu_{m^x_j}:j\geq 1\right)\right).$ To prove measurability of $(x,\mu)\mapsto i(\mu)(x)$, it is enough to prove that mappings  $(x,\mu)\mapsto \mu_{m^x_j}\in \mathcal{P}(M)$ are measurable. This follows from the measurability of $x\mapsto m^x_j$ and the equality 
$$
\{(x,\mu):\mu_{m^x_j}\in B\}=\bigcup_{r\in I}\{(x,\mu):m^x_j=r, \ \mu_r\in B\}, \ B\in \mathcal{B}(\mathcal{P}(M)).
$$
Measurability of $e$ is obvious. Further, if $x=z_m,$ then $m^x_j=m$ as soon as $\varepsilon_j/2\leq \min_{n\in I, n<m}\rho(z_m,z_n).$  So, $i(\mu)(z_m)=\mu_m$ and
$e(i(\mu))_m=i(\mu)(z_m)=\mu_m.$
In particular, $\mathfrak{p}\circ \mathfrak{p}=\mathfrak{p}.$ 
Equality $\mathfrak{p}(K)(x)=i(e(K))(x)$ proves  measurability of the mapping $(K,x)\mapsto \mathfrak{p}(K)(x).$

\end{proof}

For $n\geq 1$ define mappings $\Phi_n:(\mathcal{P}(M)^I)^n\to E,$ $\Psi_n : M\times (\mathcal{P}(M)^I)^n\to \mathcal{P}(M)$ by formulas 
$$
\Phi_n(\mu^1,\ldots,\mu^n)(x)=i(\mu^1)\ldots i(\mu^n)(x)=\Psi_n(x,\mu^1,\ldots,\mu^n).
$$

\begin{lem} \label{lem:1.5.2} For all $n\geq 1$ mappings $\Phi_n$ and $\Psi_n$ are well-defined and measurable.
\end{lem}

\begin{proof} We note that the mapping $(\mu,K)\mapsto \mu\mathfrak{p}(K)$ is measurable. By induction, it follows that 
$$
\Psi_n(x,\mu^1,\ldots,\mu^n)=\Psi_{n-1}(x,\mu^2,\ldots,\mu^n)i(\mu^n)=\Psi_{n-1}(x,\mu^2,\ldots,\mu^n)\mathfrak{p}(i(\mu^n))
$$
is measurable. 

\end{proof}

%
%
%
%

\subsection{Probability measures $\Pi_t$}\label{subsec:Pi_t} By Kolmogorov's theorem, for every $t\geq 0$ there exists a unique probability measure $\Pi_t$ on $\mathcal{P}(M)^I,$ such that for any finite set $J\subset I$ and $B\in \mathcal{B}(\mathcal{P}(M)^{|J|})$
$$
\Pi_t\{\mu: \mu|_J\in B\}=\Pi^{(|J|)}_t((z_m)_{m\in J},B).
$$

%
%
%

\begin{prp}\label{prop:fdd}
For any $(i_1,\ldots,i_n)\in I^n$ and $B\in \mathcal{B}(\mathcal{P}(M)^{n})$
$$
\Pi_t\{\mu: (\mu_{i_1},\ldots,\mu_{i_n})\in B\}=\Pi^{(n)}_t((z_{i_1},\ldots,z_{i_n}),B).
$$

\end{prp}

\begin{rem}
\label{rem:coincidence_of_indices} Note that some  indices among  $i_1,\ldots,i_n$ may coincide.
\end{rem}

\begin{proof} The proof follows from statements (1) and (2) of Lemma \ref{lem:properties_of_Pi_1}.

 Let $\bigcup^n_{j=1}\{i_j\}=\{k_1,\ldots,k_p\}\subset I$ with $k_1<\ldots<k_p.$ Denote $J_l=\{j\in \{1,\ldots,n\}:i_j=k_l\},$ $1\leq l\leq p.$ Then $J_1,\ldots,J_p$ is a partition of $\{1,\ldots,n\}$ into non-empty subsets. Let $\sigma\in S_{p,k_p}$ be the injection $\sigma(l)=k_l,$ $1\leq l\leq p.$

Consider the mapping $h:\mathcal{P}(M)^p\to \mathcal{P}(M)^n$ given by 
$$
h(\mu)_j=\mu_l, \ j\in J_l, 1\leq l\leq p.
$$

Take $B=\prod^n_{j=1}B_j,$ where $B_j\in \mathcal{B}(\mathcal{P}(M)),$ $1\leq j\leq n.$  

The equality $h(\mu_{k_1},\ldots,\mu_{k_p})=(\mu_{i_1},\ldots,\mu_{i_n})$ implies 

$$
\begin{aligned}
\Pi_t\{\mu: (\mu_{i_1},\ldots,\mu_{i_n})\in B\}&=\Pi_t\{\mu: (\mu_{k_1},\ldots,\mu_{k_p})\in h^{-1}(B)\}\\
&=\Pi^{(p)}_t\left((z_{k_1},\ldots,z_{k_p}), h^{-1}(B)\right).
\end{aligned}
$$

For every $l\in \{1,\ldots, p\}$ choose  $j(l)\in J_l,$ and set $C_l=\cap_{j\in J_l}B_j.$  Consider injections  $\alpha\in S_{p,n},$ $\alpha(l)=j(l),$ $1\leq l\leq p,$ and $\beta_{j_1,j_2;l}\in S_{2,n},$ $\beta_{j_1,j_2;l}(i)=j_i,$ $i=1,2.$ Here $j_1,j_2\in J_l,$ $j_1\ne j_2.$ We note that
$$
\begin{aligned}
\Pi^{(n)}_t((z_{i_1},\ldots,z_{i_n}),\beta^{-1}_{j_1,j_2;l}(\mathcal{P}(M)^2\setminus \Delta))&=\Pi^{(2)}((z_{i_{j_1}},z_{i_{j_2}}),\mathcal{P}(M)^2\setminus \Delta)\\
&=\Pi^{(2)}((z_{k_l},z_{k_l}),\mathcal{P}(M)^2\setminus \Delta)=0.
\end{aligned}
$$
So,
$$
\begin{aligned}
\Pi^{(n)}_t&((z_{i_1},\ldots,z_{i_n}),B)=\Pi^{(n)}_t\left((z_{i_1},\ldots,z_{i_n}),B\cap \left(\bigcap^p_{l=1}\bigcap_{\substack{(j_1,j_2)\in J^2_l \\ j_1\ne j_2}}\beta^{-1}_{j_1,j_2;l}(\Delta)\right)\right)\\
&=\Pi^{(n)}_t\left((z_{i_1},\ldots,z_{i_n}),\alpha^{-1}\left(\prod^p_{l=1} C_l\right)\cap \left(\bigcap^p_{l=1}\bigcap_{\substack{(j_1,j_2)\in J^2_l\\ j_1\ne j_2}}\beta^{-1}_{j_1,j_2;l}(\Delta)\right)\right)\\
&=\Pi^{(n)}_t\left((z_{i_1},\ldots,z_{i_n}),\alpha^{-1}\left(\prod^p_{l=1} C_l\right)\right)=\Pi^{(p)}_t\left((z_{k_1},\ldots,z_{k_p}),\prod^p_{l=1} C_l\right)\\
&=\Pi^{(p)}_t\left((z_{k_1},\ldots,z_{k_p}),h^{-1}(B)\right)=\Pi_t\{\mu: (\mu_{i_1},\ldots,\mu_{i_n})\in B\}.
\end{aligned}
$$

\end{proof}

The measure $\Pi_t$ must be understood as the distribution of $(K_{0,t}(z): z\in Z).$ We will recover the distribution  $\nu_t$ approximating the distribution of $K_{0,t}(x)$ by distributions of $(K_{0,t}(z_{m^x_j}): j\geq 1),$ where $m^x_j$ was defined in \eqref{eq:approximating_sequence}. To do this we need several estimates on the speed of approximation.

\begin{lem}
\label{lem:aux1} Let $C\subset M$ be compact and $t\geq 0.$ There exists $j_0\geq 1$ such that for all $j\geq j_0$ and all  $x\in C$
$$
\Pi_t\{\mu: d(\mu_{m^x_j},\mu_{m^x_{j+1}})\geq 2^{-j}\}\leq 2^{-j}.
$$
\end{lem}

\begin{proof} There is $l\geq 1$ such that $\{u\in M: \rho(u,C)\leq 1\}\subset L_l.$ Take $j_0\geq t\vee l$ such that $\varepsilon_{j_0}<1.$ If $x\in C$ and $j\geq j_0,$ then 
$$
\rho(z_{m^x_j},x)<\frac{\varepsilon_j}{2}<1, \  \rho(z_{m^x_{j+1}},x)<\frac{\varepsilon_{j+1}}{2}<1.
$$
So, $(t,z_{m^x_j},z_{m^x_{j+1}})\in [0,j]\times L^2_j.$ Since $\rho(z_{m^x_j},z_{m^x_{j+1}})<\varepsilon_j,$ we deduce that 
$$
\Pi_t\{\mu: d(\mu_{m^x_j},\mu_{m^x_{j+1}})\geq 2^{-j}\}=\Pi^{(2)}_t((z_{m^x_j},z_{m^x_{j+1}}),\Delta^c_{2^{-j}})\leq 2^{-j}.
$$
\end{proof}


\begin{lem}\label{lem:1.5.7} For all $x\in M$ and $\Pi_t$-a.a $\mu\in \mathcal{P}(M)^I,$ 
$$
\lim_{j\to\infty}\mu_{m^x_j}=i(\mu)(x).
$$


\end{lem}

\begin{proof} 
By the Lemma \ref{lem:aux1}, for all $j\geq j_0$ 
$$
\Pi_t\{\mu: d(\mu_{m^x_j},\mu_{m^x_{j+1}})\geq 2^{-j}\}\leq 2^{-j}.
$$
By the Borel-Cantelli lemma, for $\Pi_t$-a.a $\mu\in \mathcal{P}(M)^I,$ $\sum^\infty_{j=1}d(\mu_{m^x_j},\mu_{m^x_{j+1}})<\infty.$  So, for $\Pi_t$-a.a. $\mu\in \mathcal{P}(M)^I$ the limit $\lim_{j\to\infty}\mu_{m^x_j}$   exists and  necessarily coincides with $i(\mu)(x).$

\end{proof}

\subsection{Feller convolution semigroup $(\nu_t:t\geq 0)$.} \label{subsec:FCS_nu_t} Define $\nu_t=\Pi_t\circ i^{-1}.$ $\nu_t$ is a regular probability measure on $(E,\mathcal{E})$ with measurable presentation $\mathfrak{p}.$ Indeed, the mapping
$$ 
\mathfrak{p}\left(K\right)(x)=i(e(K))(x)
$$
is  $\mathcal{E}\otimes \mathcal{B}(M)/\mathcal{B}(M)$-measurable (Lemma \ref{lem:measurable_presentation}). Further,  for every $x\in M$ 
$$
\begin{aligned}
\nu_t\{K: \mathfrak{p}(K)(x)  =K(x)\}&=\nu_t\{K: i(e(K))(x)=K(x)\}\\
&=\Pi_t\{\mu: i(e(i(\mu))(x)=i(\mu)(x)\}\\
&=\Pi_t\{\mu: i(\mu)(x)=i(\mu)(x)\}=1,
\end{aligned}
$$
since $e\circ i$ is the identity mapping on $\mathcal{P}(M)^I$ (Lemma \ref{lem:measurable_presentation}).

Consider  $x\in M^N,$ $t\geq 0,$ and $f\in C_0(M^N).$  Using Proposition \ref{prop:fdd}, Lemma \ref{lem:1.5.7},  dominated convergence theorem and the Feller property of $(\mathsf{P}^{(n)}: n\geq 1)$, we obtain
\begin{equation}
\label{eq:fdd_verification}
\begin{aligned}
\int_{E}&\left(\int_{M^N}f(y)\left(\otimes^N_{r=1}K(x_{r})\right)(dy)\right)\nu_t(dK)\\
&=\int_{\mathcal{P}(M)^I}\left(\int_{M^N}f(y)\left(\otimes^N_{r=1}i(\mu)(x_{r})\right)(dy)\right)\Pi_t(d\mu)\\
&=\lim_{j\to\infty}\int_{\mathcal{P}(M)^I}\left(\int_{M^N}f(y)\left(\otimes^N_{r=1}\mu_{m^{x_{r}}_j}\right)(dy)\right)\Pi_t(d\mu)\\
&=\lim_{j\to\infty}\int_{\mathcal{P}(M)^N}\left(\int_{M^N}f(y)\left(\otimes^N_{r=1}\mu_r\right)(dy)\right)\Pi^{(N)}_t((z_{m^{x_{1}}_j},\ldots,z_{m^{x_{N}}_j}),d\mu)\\
&=\lim_{j\to\infty} \mathsf{P}^{(N)}_tf(z_{m^{x_{1}}_j},\ldots,z_{m^{x_{N}}_j})=\mathsf{P}^{(N)}_tf(x_{1},\ldots,x_{N}).
\end{aligned}
\end{equation}

Now we can verify that $(\nu_t : t\geq 0)$ is the needed  Feller convolution semigroup in the space of kernels on $M$. Let $t,s\geq 0.$ From the Lemma \ref{CSFTF_determines_FCS} it is enough to verify that integrals of functions 
$$
K\mapsto \int_{M^N}f(y)\left(K(x_{i_1})\otimes \ldots\otimes K(x_{i_N})\right)(dy),
$$
where  $x\in M^n,$ $f\in C_0(M^N),$ $(i_1,\ldots,i_N)\in \{1,\ldots,n\}^N,$ coincide for distributions $\nu_t*\nu_s$ and $\nu_{t+s}.$ Using Fubini's theorem, we have
$$
\begin{aligned}
\int_E &\left(\int_{M^N}f(y)\left(\otimes^N_{r=1} K(x_{i_r})\right)(dy)\right)(\nu_t*\nu_s)(dK)\\
&=\int_E\int_E \left(\int_{M^N}f(y)\left(\otimes^N_{r=1}K_1\mathfrak{p}(K_2)(x_{i_r})\right)(dy)\right)\nu_t(dK_1)\nu_s(dK_2)\\
&=\int_E\int_E \left(\int_{M^N}\int_{M^N}f(z)\left(\otimes^N_{r=1}\mathfrak{p}(K_2)(y_r)\right)(dz)\left(\otimes^N_{r=1}K_1(x_{i_r})\right)(dy)\right)\nu_t(dK_1)\nu_s(dK_2)\\
&=\int_E\int_{M^NN}P^{(N)}_sf(y)\left(\otimes^N_{r=1}K_1(x_{i_r})\right)(dy)\nu_t(dK_1)=P^{(N)}_{t}P^{(N)}_sf(x_{i_1},\ldots,x_{i_N})\\
&=P^{(N)}_{t+s}f(x_{i_1},\ldots,x_{i_N})=\int_E \left(\int_{M^N}f(y)\left(\otimes^N_{r=1}K(x_{i_r})\right)(dy)\right)\nu_{t+s}(dK).
\end{aligned}
$$
The equality $\nu_t*\nu_s=\nu_{t+s}$ is proved.

We  verify  conditions (2) and (3) of the Definition \ref{def:FCS_kernel}. Let $f\in C_0(M)$ and $\varepsilon>0.$  Then
$$
\begin{aligned}
\sup_{x\in M}\nu_t&\left\{K: |Kf(x)-f(x)|\geq \varepsilon\right\}\leq \varepsilon^{-2}\sup_{x\in M}\int_E \left(Kf(x)-f(x)\right)^2\nu_t (dK)\\
&=\varepsilon^{-2}\sup_{x\in M}\left(\mathsf{P}^{(2)}_tf^{\otimes 2}(x,x)-2f(x)\mathsf{P}^{(1)}f(x)+f^2(x)\right)\to 0,
\end{aligned}
$$
as $t\to 0+.$  Further, 
$$
\begin{aligned}
\nu_t&\left\{K: |Kf(y)-Kf(x)|\geq \varepsilon\right\}\leq \varepsilon^{-2}\int_E \left(Kf(y)-Kf(x)\right)^2\nu_t(dK)\\
&=\varepsilon^{-2}\left(\mathsf{P}^{(2)}_tf^{\otimes 2}(y,y)-2\mathsf{P}^{(2)}_tf^{\otimes 2}(x,y)+\mathsf{P}^{(2)}_tf^{\otimes 2}(x,x)\right)\to 0, y\to x.
\end{aligned}
$$
Finally,
$$
\begin{aligned}
\nu_t&\left\{K: |Kf(y)|\geq \varepsilon\right\}\leq \varepsilon^{-2}\int_E \left(Kf(y)\right)^2\nu_t(dK)=\varepsilon^{-2}\mathsf{P}^{(2)}_tf^{\otimes 2}(y,y)\to 0, y\to \infty.\\
\end{aligned}
$$

%

Equation \eqref{eq:fdd_verification} implies that the consistent sequence of Feller transition functions that corresponds to $(\nu_t:t\geq 0)$ is exactly $(\mathsf{P}^{(n)}:n\geq 1).$ This finishes the proof of Theorem \ref{thm:FCS_kernels}.

In the next section we will need the following result. 

\begin{lem}\label{lem:1.5.14} For all $t_1,\ldots,t_n\geq 0$
$$
(\Pi_{t_1}\otimes \ldots\otimes \Pi_{t_n})\circ \Phi^{-1}_n=\nu_{t_1+\ldots+t_n}.
$$

\end{lem}

\begin{proof} The proof is by induction on $n\geq 1.$ For $n=1$ the statement is the definition of $\nu_t.$  Assume the result is proved for $n-1$ and let $A\in \mathcal{E}.$ We note that the map $(K,\mu)\mapsto Ki(\mu)=K\mathfrak{p}(i(\mu))$ is $\mathcal{E}\times \mathcal{B}(\mathcal{P}(M)^I)/\mathcal{E}$-measurable. Using Fubini's theorem, we get 
$$
\begin{aligned}
(&\Pi_{t_1}\otimes \ldots \otimes \Pi_{t_n})\left(\Phi^{-1}_n(A)\right)\\
&=\left(\Pi_{t_1}\otimes \ldots \otimes \Pi_{t_n}\right)\left\{(\mu^1,\ldots,\mu^n)\in (\mathcal{P}(M)^I)^n: i(\mu^1)\ldots i(\mu^{n-1}) \mathfrak{p}(i(\mu^n))\in A\right\}\\
&=\int_{E}\nu_{t_1+\ldots+t_{n-1}}\{K_1: K_1\mathfrak{p}(K_2)\in A\}\nu_{t_n}(dK_2)=(\nu_{t_1+\ldots+t_{n-1}}*\nu_{t_n})(A)\\
&=\nu_{t_1+\ldots+t_n}(A).
\end{aligned}
$$

\end{proof}

\section{Proof of the Theorem \ref{thm:SFK}}
\label{sec:main_thm_2}

\subsection{Probability space $(\Omega,\mathcal{A},\mathsf{P})$}\label{subsec:prob_space} As before, $Z$ is an at most countable dense set in $M$ and $m\mapsto z_m$ is a bijection between a subset  $I\subset \mathbb{N}$ and the set $Z$.  Recall a probability measure $\Pi_t$  on $(\mathcal{P}(M)^I,\mathcal{B}(\mathcal{P}(M))^{\otimes I})$ constructed in the Section \ref{subsec:Pi_t}. 
We will use mappings $i:\mathcal{P}(M)^I\to E,$ $e:E\to \mathcal{P}(M)^I,$ $\Phi_n:(\mathcal{P}(M)^I)^n\to E,$ $\Psi_n:M\times (\mathcal{P}(M)^I)^n\to \mathcal{P}(M),$ $\mathfrak{p}:E\to E,$ defined in Section \ref{subsec:approximating_procedure}. 
We recall  that  $\mathfrak{p}=i\circ e$ is a measurable presentation of every measure $\nu_t$ (Section \ref{subsec:FCS_nu_t}) and that $\mathfrak{p}\circ \mathfrak{p}=\mathfrak{p}$ (Lemma \ref{lem:measurable_presentation}).

For each $n\geq 0$ consider the probability space 
$$
(S_n,\mathcal{S}_n,\mathsf{P}_n)=(\mathcal{P}(M)^I,\mathcal{B}(\mathcal{P}(M))^{\otimes I},\Pi_{2^{-n}})^{\otimes \mathbb{Z}}.
$$
Note that $\mathcal{S}_n$ is the Borel $\sigma$-field on the complete separable metric space $S_n.$ Denote $D_n=2^{-n}\mathbb{Z},$ $D=\bigcup^\infty_{n=0} D_n.$ 

\begin{rem}
\label{rem:idea} If $\omega^n\in S_n$ we intuitively understand  $i(\omega^n_l)$ as the random kernel $K_{l2^{-n},(l+1)2^{-n}}$ from the future flow.
\end{rem}

Consider mappings
$$
\pi_{n-1,n}:S_n\to S_{n-1}, \ \pi_{n-1,n}\left(\omega^n\right)=\left(e\left( i(\omega^n_{2l})i(\omega^n_{2l+1})\right)\right)_{l\in \mathbb{Z}}.
$$

Mappings $\pi_{n-1,n}$ are measurable and surjective. To show measurability we note that the $l-$th component of $\pi_{n-1,n}$ equals $e\left( i(\omega^n_{2l})i(\omega^n_{2l+1})\right)\in \mathcal{P}(M)^I.$ Its element that corresponds to $m\in I$ is
$$
i(\omega^n_{2l})i(\omega^n_{2l+1})(z_m)=\Psi_2\left(z_m,\omega^n_{2l},\omega^n_{2l+1}\right)\in \mathcal{P}(M),
$$
and the mapping $\Psi_2$ is measurable (Lemma \ref{lem:1.5.2}).  Surjectivity of $\pi_{n-1,n}$ follows from the following Lemma.

\begin{lem}
\label{lem:surjectivity} Consider $\mu_0=(\delta_{z_m})_{m\in I}.$ Then for each $x\in M,$ $i(\mu_0)(x)=\delta_x.$ In particular, kernel $x\mapsto \delta_x$ is invariant under $\mathfrak{p}.$

\end{lem}

\begin{proof} For each $x\in M$ we have $\lim_{j\to\infty}z_{m^x_j}=x.$  Hence,
$$
i(\mu_0)(x)=\ell\left(\left(\delta_{z_{m^x_j}}:j\geq 1\right)\right)=\delta_x.
$$ 
Denote $K_0(x)=\delta_x.$ Then 
$$
\mathfrak{p}(K_0)(x)=i(\mu_0)(x)=\delta_x=K_0(x).
$$
\end{proof}

From Lemma \ref{lem:surjectivity} we deduce that   $i(\mu_0)K=K$ for each kernel $K\in E.$ For given  $\omega^{n-1}\in S_{n-1}$ define $\omega^n_{2l}=\mu_0,$ $\omega^n_{2l+1}=\omega^{n-1}_l.$ Then 
$$
(\pi_{n-1,n}(\omega^n))_l=e\left(i(\mu_0)i(\omega^{n-1}_{l})\right)=e\circ i (\omega^{n-1}_{l})=\omega^{n-1}_l.
$$
This proves surjectivity of $\pi_{n-1,n}.$ 

We note that  $\mathsf{P}_n\circ \pi^{-1}_{n-1,n}=\mathsf{P}_{n-1}.$ Indeed, under the measure $\mathsf{P}_n\circ \pi^{-1}_{n-1,n}$ components of $\omega^{n-1}$ are independent, and the law of $\omega^{n-1}_l$ equals  (Lemma \ref{lem:1.5.14})
$$
\left(\Pi_{2^{-n}}\otimes \Pi_{2^{-n}}\right)\circ (e\circ \Phi_2)^{-1}=\nu_{2^{-(n-1)}}\circ e^{-1}=\Pi_{2^{-(n-1)}}\circ i^{-1}\circ e^{-1}=\Pi_{2^{-(n-1)}}.
$$

Let the set $\Omega$ be the inverse limit 
$$
\Omega=\left\{\omega=(\omega^n)_{n\geq 0}\in \prod^\infty_{n=0}S_n: \ \forall n\geq 1 \ \pi_{n-1,n}(\omega^{n})=\omega^{n-1}\right\}
$$
(in the terminology of K.R. Parthasarathy \cite[Sec. 2, Ch. V]{Parthasarathy}). Let the mapping $\pi_n:\Omega\to S_n$ be a projection, $\pi_n(\omega)=\omega^n,$ and the $\sigma$-field $\mathcal{A}$ on $\Omega$ be  the smallest $\sigma$-field under which all projections $\pi_n,$ $n\geq 0,$ are measurable.
%
%
%
%
%
%
There exists a unique probability measure $\mathsf{P}$ on $(\Omega,\mathcal{A}),$ such that for all $n\geq 0$ and $C\in\mathcal{S}_n,$
$$
\mathsf{P}(\pi^{-1}_n(C))=\mathsf{P}_n(C)
$$
\cite[Th. 3.2, Ch. V]{Parthasarathy}.

For $(s,t)\in D^2,$ $s\leq t,$ let $\mathcal{A}_{s,t}$ be the $\sigma$-field generated by mappings $\omega\mapsto \omega^n_u,$ where $n\geq 0$ and  $u\in\mathbb{Z}$ are such that $(s,t)\in D^2_n$ and  $u2^{-n}\in [s,t).$ We note that $\mathcal{A}_{s,s}$ is the trivial $\sigma$-field $\{\varnothing,\Omega\}.$

\begin{lem}
\label{lem:measurable_functions_of_omega} For all $0\leq n\leq k$ and any $u\in \mathbb{Z},$ $\omega^n_u$ is a measurable function of  $\omega^k_{2^{k-n}u},\ldots,\omega^k_{2^{k-n}u+2^{k-n}-1}.$
\end{lem}

\begin{proof} The proof is by induction on $k-n\geq 0.$ If $k=n,$ then the statement is obvious. Assume that $k< n$ and that  the statement is proved for  $k-1-n.$ Let   $u\in \mathbb{Z}.$ By the inductive hypothesis, there exists a measurable function $F:(\mathcal{P}(M)^I)^{2^{k-1-n}}\to \mathcal{P}(M)^I,$ such that 
$$
\omega^n_u=F\left(\omega^{k-1}_{2^{k-1-n}u},\ldots,\omega^{k-1}_{2^{k-1-n}u+2^{k-1-n}-1}\right).
$$
Then 
$$
\omega^n_u=F\left(e\left(i\left(\omega^{k}_{2^{k-n}u}\right)i\left(\omega^{k}_{2^{k-n}u+1}\right)\right),\ldots,e\left( i\left(\omega^{k}_{2^{k-n}u+2^{k-n}-2}\right) i\left(\omega^{k}_{2^{k-n}u+2^{k-n}-1}\right)\right)\right)
$$
is a measurable function of $\omega^k_{2^{k-n}u},\ldots,\omega^k_{2^{k-n}u+2^{k-n}-1}.$

\end{proof}

\begin{lem}
\label{lem:sigma_fields_A_st}  If $(r,s,t)\in D^3,$ $r\leq s\leq t,$ then $\sigma-$fields $\mathcal{A}_{r,s}$ and $\mathcal{A}_{s,t}$ are independent. If $(r_1,r_2,r_3,r_4)\in D^4,$ $r_1\leq r_2\leq r_3\leq r_4,$ then $\mathcal{A}_{r_2,r_3}\subset \mathcal{A}_{r_1,r_4}.$
\end{lem}
\begin{proof}
Let $(r,s,t)\in D^3,$ $r<s<t.$ Consider $n_1,\ldots,n_k\geq 0,$ $u_1,\ldots,u_k\in \mathbb{Z},$  $m_1,\ldots,m_l\geq 0,$ $v_1,\ldots,v_l\in \mathbb{Z},$ such that $(r,s)\in D^2_{n_i},$ $u_12^{-n_1},\ldots,u_k 2^{-n_k}\in [r,s),$   $(s,t)\in  D^2_{m_j},$ $v_12^{-m_1},\ldots,v_l 2^{-m_l}\in [s,t).$ Denote $N=\max(n_1,\ldots,n_k,m_1,\ldots,m_l).$ By the Lemma \ref{lem:measurable_functions_of_omega}, each $\omega^{n_i}_{u_i}$ is a measurable function of $\omega^N_{2^{N-n_i}u_i},\ldots,\omega^N_{2^{N-n_i}u_i+2^{N-n_i}-1},$ and each $\omega^{m_j}_{v_j}$ is a measurable function of $\omega^N_{2^{N-m_j}v_j},\ldots,\omega^N_{2^{N-m_j}v_j+2^{N-m_j}-1}.$  Since $s\in D_{n_i},$ we write $s=2^{-n_i}a,$ $a\in \mathbb{Z}.$ Inequality $u_i 2^{-n_i}<s$ implies $u_i+1\leq a.$ Hence,
$$
2^{N-n_i}u_i+2^{N-n_i}-1\leq 2^{N-n_i}a-1=2^Ns-1.
$$
Since $s\in D_{m_j},$ we write $s=2^{-m_j}b,$ $b\in \mathbb{Z}.$ Inequality $v_j 2^{-m_j}\geq s$ implies $v_j\geq b.$ Hence,
$$
2^{N-m_j}v_j\geq 2^{N-m_j}b=2^Ns.
$$
Independence of $\mathcal{A}_{r,s}$ and $\mathcal{A}_{s,t}$ now follows from idependence of $(\omega^N_k)_{k\in \mathbb{Z}}.$

Let  $(r_1,r_2,r_3,r_4)\in D^4,$ $r_1\leq r_2< r_3\leq r_4.$ Let $(r_2,r_3)\in D^2_n,$ $u2^{-n}\in [r_2,r_3).$  Take $N\geq n,$ such that $(r_1,r_2,r_3,r_4)\in D^4_N.$ $\omega^n_u$ is a measurable function of $\omega^N_{2^{N-n}u},$ $\ldots,$ $\omega^N_{2^{N-n}u+2^{N-n}-1}.$ We note that 
$$
2^{N-n}u 2^{-N}=2^{-n}u\geq r_2\geq r_1,
$$
and, since $u+1\leq 2^nr_3,$
$$
\left(2^{N-n}u+2^{N-n}-1\right)2^{-N}=2^{-n}(u+1)-2^{-N}\leq r_3-2^{-N}<r_3\leq r_4.
$$
It follows that $\omega^n_u$ is $\mathcal{A}_{r_1,r_4}$-measurable.

\end{proof}

\subsection{Random kernels $^DK_{s,t}$ for $(s,t)\in D^2,$ $s\leq t$} 
\label{subsec:SFK_dyadic}

%

\begin{dfn}
\label{def:1.6.1}  \begin{enumerate} \item  For $(s,t)\in D^2_n,$ $s<t,$ and $\omega\in \Omega,$ we define 
$$
{}^DK^{(n)}_{s,t}(\omega)=\Phi_{(t-s)2^n}\left(\omega^n_{s2^n},\ldots,\omega^n_{t2^n-1}\right).
$$

\item For $(s,t)\in D^2,$ $s<t,$ we define 
$$
{}^DK_{s,t}=\mathfrak{p}\left({}^DK^{(n)}_{s,t}\right),
$$
where $n\geq 0$ is the minimal non-negative integer, such that $(s,t)\in D^2_n.$

\item For all $t\in D$ and $x\in M,$ we define 
$$
{}^D K_{t,t}(x)=\delta_x.
$$

\end{enumerate}

\end{dfn}


\begin{prp}
\label{prp:properties_of_K} \begin{enumerate} \item 
For all $(s,t)\in D^2,$ $s\leq t,$ we have $\mathfrak{p}({}^D K_{s,t})={}^DK_{s,t}.$

\item  Let $s\in D_n.$ Then ${}^DK_{s,s+2^{-n}}\left(\omega\right)={}^D K^{(n)}_{s,s+2^{-n}}\left(\omega\right)=\mathfrak{p} \left({}^D K^{(n)}_{s,s+2^{-n}}(\omega)\right)=i\left(\omega^n_{s2^n}\right),$ $\omega\in \Omega.$

\item If $(s,t)\in D^2_n,$ $s< t,$ then ${}^D K^{(n)}_{s,t}$  is a measurable function of ${}^DK^{(n)}_{s,s+2^{-n}},$ $\ldots,$${}^DK^{(n)}_{t-2^{-n}t},$ and is  $\mathcal{A}_{s,t}/\mathcal{E}$-measurable.  If $(s,t)\in D^2,$ $s\leq t,$ then ${}^DK_{s,t}$  is a measurable function of 
$$
\left\{{}^DK^{(n)}_{r,r+2^{-n}}: n\geq 0, r\in D_n, [r,r+2^{-n})\subset [s,t)\right\}
$$
 and  is $\mathcal{A}_{s,t}/\mathcal{E}$-measurable.

\item If $(s,t)\in D^2_n,$ $s\leq t,$ then the distribution of ${}^D K^{(n)}_{s,t}$ in the space of kernels $(E,\mathcal{E})$ coincides with $\nu_{t-s}.$  If $(s,t)\in D^2,$ $s\leq t,$ then the distribution of ${}^D K_{s,t}$ in the space of kernels $(E,\mathcal{E})$ coincides with $\nu_{t-s}.$

\item If $(r,s,t)\in D^3_n,$ $r< s< t,$  then ${}^DK^{(n)}_{r,t}={}^DK^{(n)}_{r,s} {}^DK^{(n)}_{s,t}.$ 


\end{enumerate}
\end{prp}

\begin{proof}

\begin{enumerate}

\item 
 If $s<t,$ then the result follows from the fact that $\mathfrak{p}\circ \mathfrak{p}=\mathfrak{p}.$ Let $s=t.$  Let $\mu_0:=e({}^DK_{t,t})=(\delta_{z_m})_{m\in I}.$ Then $i(\mu_0)={}^DK_{t,t}$ (Lemma \ref{lem:surjectivity}) and, since $e\circ i$ is the identity mapping on $\mathcal{P}(M)^I,$
 $$
 \mathfrak{p}({}^DK_{t,t})=i\circ e\circ i (\mu_0)=i(\mu_0)={}^DK_{t,t}.
 $$

\item  Let $m=\inf\{k\geq 0: (s,s+2^{-n})\in D^2_k\}\leq n.$ Write $s=j2^{-n}$ with $j\in \mathbb{Z},$  $s+2^{-n}=(j+1)2^{-n}.$ Assume that $m<n.$ Then $j2^{-n}=k2^{-m},$ $(j+1)2^{-n}=l2^{-m},$  $(l-k)2^{-m}=2^{-n}.$ It follows that  $l-k=2^{m-n}\in (0,1),$ which is impossible. So, $m=n.$ 
Further,
$$
{}^DK_{s,s+2^{-n}}\left(\omega\right)=\mathfrak{p}\left({}^DK^{(n)}_{s,s+2^{-n}}\left(\omega\right)\right)=\mathfrak{p}\circ i (\omega^n_{s2^n})=i (\omega^n_{s2^n}).
$$
Here we again use the property $\mathfrak{p}\circ i=i.$

\item 
Random mapping $K^{(n)}_{s,s+2^{-n}}(\omega)=i(\omega^n_{s2^n})$ is $\mathcal{A}_{s,s+2^{-n}}/\mathcal{E}$-measurable. The needed result follows from equality
$$
{}^D K^{(n)}_{s,t}=\Phi_{(t-s)2^n}\left(e\left({}^DK^{(n)}_{s,s+2^{-n}}\right),\ldots,e\left({}^DK^{(n)}_{t-2^{-n},t}\right)\right).
$$


\item Follows from the Lemma \ref{lem:1.5.14}, definitions of ${}^DK^{(n)}_{s,t}$ and ${}^DK_{s,t},$ and the fact that distributions of $K$ and $\mathfrak{p}(K)$ with respect to any measure $\nu_t$ coincide.

\item Follows from the definition of ${}^DK^{(n)}.$ 
%
%



\end{enumerate}

\end{proof}

%
%
%
%
%
%
%
%
%
%
%
%
%
%

\begin{lem}
\label{lem:1.6.4} Let $(s,t)\in D^2_n,$ $s<t.$ For any $\mathcal{P}(M)$-valued random element $\mathcal{M},$ $\mathcal{M}{}^DK^{(n)}_{s,t}$ and $\mathcal{M}{}^DK_{s,t}$ are  random elements in $\mathcal{P}(M).$ 

\end{lem}

\begin{proof} We note that the mapping $(\mu,K)\mapsto \mu\mathfrak{p}(K)$ is measurable. This implies measurability of $\mathcal{M}{}^DK_{s,t},$ since $\mathfrak{p}({}^DK_{s,t})={}^DK_{s,t}.$ If $t=s+2^{-n},$ we have $\mathfrak{p}({}^DK^{(n)}_{s,s+2^{-n}})={}^DK^{(n)}_{s,s+2^{-n}}.$ So, $\mathcal{M}{}^DK^{(n)}_{s,s+2^{-n}}$ is a random element in $\mathcal{P}(M).$ By induction, 
$$
\mathcal{M}{}^DK^{(n)}_{s,t}=\mathcal{M}{}^DK^{(n)}_{s,t-2^{-n}}{}^DK^{(n)}_{t-2^{-n},t}
$$
is a random element in $\mathcal{P}(M).$ 

\end{proof}

%
%
%
%

\begin{prp}
\label{prop:1.6.3_aux} Let $(s,t)\in D^2_n,$ $s<  t.$ For any $\mathcal{P}(M)$-valued random element $\mathcal{M}$ independent from $\mathcal{A}_{s,t},$
 $$
\mathcal{M}{}^DK^{(n)}_{s,t}=\mathcal{M}{}^DK_{s,t} \ \mbox{ a.s.}
$$

\end{prp}

\begin{proof} Denote by $\Pi$ the distribution of $\mathcal{M}$ in $\mathcal{P}(M).$   
%
We show that  $\mathcal{M}{}^DK^{(n)}_{s,t}=\mathcal{M}{}^DK^{(n+1)}_{s,t}$ a.s. 
For every $x\in M$ statements (1) and (2) of Proposition \ref{prp:properties_of_K} imply

$$
\begin{aligned}
{}^DK^{(n)}_{s,s+2^{-n}}(\omega)(x)&=i(\omega^n_{s2^n})(x)=i\circ e\left( i(\omega^{n+1}_{s2^{n+1}})i(\omega^{n+1}_{s2^{n+1}+1})\right)(x)\\
&=\mathfrak{p}\left(\Phi_2\left(\omega^{n+1}_{s2^{n+1}},\omega^{n+1}_{s2^{n+1}+1}\right)\right)(x)=\mathfrak{p}\left({}^DK^{(n+1)}_{s,s+2^{-n}}(\omega)\right)(x)\\
&={}^D K^{(n+1)}_{s,s+2^{-n}}(\omega)(x)\ \mbox{ a.s.}
\end{aligned}
$$
We note that 
$$
\begin{aligned}
{}^DK^{(n+1)}_{s,s+2^{-n}}(\omega)(x)&={}^DK^{(n+1)}_{s,s+2^{-n-1}}(\omega){}^DK^{(n+1)}_{s+2^{-n-1},s+2^{-n}}(\omega)(x)\\
&=\mathfrak{p}\left({}^DK^{(n+1)}_{s,s+2^{-n-1}}(\omega)\right) \mathfrak{p}\left({}^DK^{(n+1)}_{s+2^{-n-1},s+2^{-n}}(\omega)\right)(x)\\
&=\mathfrak{p}\left({}^DK^{(n+1)}_{s,s+2^{-n-1}}(\omega)\right)\mathfrak{p}\circ \mathfrak{p}\left({}^DK^{(n+1)}_{s+2^{-n-1},s+2^{-n}}(\omega)\right)(x).
\end{aligned}
$$
Mappings 
$$
(x,K_1,K_2,K_3)\mapsto 1_{\mathfrak{p}(K_1)(x)= \mathfrak{p}(K_2)\mathfrak{p}\circ \mathfrak{p}(K_3)(x)}
$$ 
and 
$$
(\mu,K_1,K_2,K_3)\mapsto \mu\{x: \mathfrak{p}(K_1)(x)= \mathfrak{p}(K_2)\mathfrak{p}\circ \mathfrak{p}(K_3)(x)\}
$$
are measurable. 
Fubini's theorem implies
$$
\begin{aligned}
\mathsf{E}&\mathcal{M}\{x: {}^DK^{(n)}_{s,s+2^{-n}}(x)={}^DK^{(n+1)}_{s,s+2^{-n}}(x)\}\\
&=\int_{\mathcal{P}(M)}\mathsf{E}\mu\{x: {}^DK^{(n)}_{s,s+2^{-n}}(x)={}^DK^{(n+1)}_{s,s+2^{-n}}(x)\}\Pi(d\mu)\\
&=\int_{\mathcal{P}(M)}\int_M \mathsf{P}\left({}^DK^{(n)}_{s,s+2^{-n}}(x)={}^DK^{(n+1)}_{s,s+2^{-n}}(x)\right)\mu(dx)\Pi(d\mu)=1.
\end{aligned}
$$
It follows that a.s. for $\mathcal{M}$-a.a. $x\in M,$
$$
{}^DK^{(n)}_{s,s+2^{-n}}(x)={}^DK^{(n+1)}_{s,s+2^{-n}}(x),
$$
and a.s.
$$
\mathcal{M}{}^DK^{(n)}_{s,s+2^{-n}}=\mathcal{M}{}^DK^{(n+1)}_{s,s+2^{-n}}.
$$

Assume the result is proved for $s+(k-1)2^{-n}=t-2^{-n}.$ Statement (5) of Proposition \ref{prp:properties_of_K} implies that
$$
{}^DK^{(n)}_{s,t}={}^DK^{(n)}_{s,t-2^{-n}}{}^DK^{(n)}_{t-2^{-n},t}, \ {}^DK^{(n+1)}_{s,t}={}^DK^{(n+1)}_{s,t-2^{-n}}{}^DK^{(n+1)}_{t-2^{-n},t}.
$$

By inductive hypothesis, a.s. 
$$
\begin{aligned}
\mathcal{M}{}^DK^{(n)}_{s,t}&=\left(\mathcal{M}{}^DK^{(n)}_{s,t-2^{-n}}\right){}^DK^{(n)}_{t-2^{-n},t}=\left(\mathcal{M}{}^DK^{(n+1)}_{s,t-2^{-n}}\right){}^DK^{(n)}_{t-2^{-n},t}\\
&=\left(\mathcal{M}{}^DK^{(n+1)}_{s,t-2^{-n}}\right){}^DK^{(n+1)}_{t-2^{-n},t}=\mathcal{M}{}^DK^{(n+1)}_{s,t}.
\end{aligned}
$$
Here we used independency of $\mathcal{M}{}^DK^{(n+1)}_{s,t-2^{-n}}$ from $\mathcal{A}_{t-2^{-n},t},$ which follows from the representation 
$$
\mathcal{M}{}^DK^{(n+1)}_{s,t-2^{-n}}=\mathcal{M}\mathfrak{p}\left({}^DK^{(n+1)}_{s,s+2^{-n-1}}\right)\ldots\mathfrak{p}\left({}^DK^{(n+1)}_{t-3\times 2^{-n-1},t-2^{-n}}\right).
$$

Mappings 
$$
(x,K_1,\ldots, K_{(t-s)2^n})\mapsto 1_{\mathfrak{p}(K_1)\ldots\mathfrak{p}(K_{(t-s)2^n})(x)=\mathfrak{p}\left( \mathfrak{p}(K_1)\ldots\mathfrak{p}(K_{(t-s)2^n})\right)(x)}
$$ 
and 
$$
(\mu,K_1,\ldots,K_{(t-s)2^n})\mapsto \mu\{x: \mathfrak{p}(K_1)\ldots\mathfrak{p}(K_{(t-s)2^n})(x)=\mathfrak{p}\left( \mathfrak{p}(K_1)\ldots\mathfrak{p}(K_{(t-s)2^n})\right)(x)\}
$$
are measurable. 
Substituting $K_j={}^D K^{(n)}_{s+(j-1)2^{-n},s+j2^{-n}},$ $1\leq j\leq (t-s)2^n,$ and using Fubini's theorem,  we get  
$$
\begin{aligned}
\mathsf{E}&\mathcal{M}\{x: {}^DK^{(n)}_{s,t}(x)=\mathfrak{p}\left({}^DK^{(n)}_{s,t}\right)(x)\}\\
&=\int_{\mathcal{P}(M)}\mathsf{E}\mu\left\{x: {}^DK^{(n)}_{s,t}(x)=\mathfrak{p}\left({}^DK^{(n)}_{s,t}\right)(x)\right\}\Pi(d\mu)\\
&=\int_{\mathcal{P}(M)}\int_M \mathsf{P}\left({}^DK^{(n)}_{s,t}(x)=\mathfrak{p}\left({}^DK^{(n)}_{s,t}\right)(x)\right)\mu(dx)\Pi(d\mu)=1.
\end{aligned}
$$
It follows that a.s. for $\mathcal{M}$-a.a. $x\in M,$
$$
{}^DK^{(n)}_{s,t}(x)=\mathfrak{p}\left({}^DK^{(n)}_{s,t}\right)(x),
$$
and a.s.
$$
\mathcal{M}{}^DK^{(n)}_{s,t}=\mathcal{M}\mathfrak{p}\left({}^DK^{(n)}_{s,t}\right).
$$
This finishes the proof of the Proposition.

\end{proof}

%
%
%

\begin{rem}
\label{rem:1.6.5} Let $(r,s,t)\in D^3,$ $r\leq s\leq t.$ For all $\mathcal{P}(M)$-valued random elements $\mathcal{M}$ independent from $\mathcal{A}_{r,t},$
$$
\mathcal{M} {}^DK_{r,t}=\mathcal{M}{}^DK_{r,s}{}^DK_{s,t} \ \mbox{ a.s.}
$$

\end{rem}

\begin{lem}
\label{lem:1.6.6_measures} Let $f\in C(\mathcal{P}(\hat{M})^2).$ For any compact $\mathcal{C}\subset \mathcal{P}(M)$ the function 
$$
(s,t,u,v,\mu,\nu)\mapsto \mathsf{E}f\left(\mu{}^DK_{s,t},\nu{}^DK_{u,v}\right)
$$
can be continuously extended to $\{(s,t)\in \mathbb{R}^2: s\leq t\}^2\times \mathcal{C}^2.$

\end{lem}

\begin{proof} By the Stone-Weierstrass theorem it is enough to consider functions $f\in C(\mathcal{P}(\hat{M})^2)$ of the form $f(\mu,\nu)=g(\mu)h(\nu),$ where $g,h\in \mathbb{A}_1(\hat{M}),$
$$
g(\varkappa)=\int_{\hat{M}^N}a(z)\varkappa^{\otimes N}(dz), \ h(\varkappa)=\int_{\hat{M}^N}b(z)\varkappa^{\otimes N}(dz),
$$
and $a,b\in C(\hat{M}^N).$ 

At first we consider the case when $a$ and $b$ have compact supports in $M^N,$ in particular  $a,b\in C_0(M^N).$ We will prove that there exists a continuous function $F:\{(s,t)\in \mathbb{R}^2: s\leq t\}^2\times \mathcal{P}(M)^2\to \mathbb{R},$ such that for all $(s,t,u,v,\mu,\nu)\in D^4\times \mathcal{P}(M)^2,$ $s\leq t,$ $u\leq v,$
$$
F(s,t,u,v,\mu,\nu)= \mathsf{E}f\left( \mu{}^DK_{s,t},\nu {}^DK_{u,v}\right)=\mathsf{E}g\left(\mu {}^DK_{s,t}\right)h\left( \nu{}^DK_{u,v}\right).
$$
Denote 
$$
\begin{aligned}
H(s,t,u,v,\mu,\nu)&=\mathsf{E}g\left(\mu {}^DK_{s,t}\right)h\left( \nu{}^DK_{u,v}\right)\\
&=\mathsf{E}\int_{M^N}{}^DK^{\otimes N}_{s,t}a (x)\mu^{\otimes N}(dx)\int_{M^N}{}^DK^{\otimes N}_{u,v}b(y)\nu^{\otimes N}(dy)\\
&=\int_{M^N}\int_{M^N}\left[\mathsf{E}{}^DK^{\otimes N}_{s,t}a (x){}^DK^{\otimes N}_{u,v}b(y)\right]\mu^{\otimes N}(dx)\nu^{\otimes N}(dy).
\end{aligned}
$$ 
We evaluate the function $H$ for different displacements of $(s,t,u,v).$

{\it Case 1:} $s\leq t\leq u\leq v.$
$$
\begin{aligned}
H(s,t,u,v,\mu,\nu)&=\int_{M^N}\int_{M^N}\left[\mathsf{E}{}^DK_{s,t}^{\otimes N}a(x){}^DK_{u,v}^{\otimes N}b(y)\right]\mu^{\otimes N}(dx)\nu^{\otimes N}(dy)\\
&=\int_{M^N}\int_{M^N}\mathsf{P}^{(N)}_{t-s}a(x)\mathsf{P}^{(N)}_{v-u}b(y)\mu^{\otimes N}(dx)\nu^{\otimes N}(dy).
\end{aligned}
$$
The right-hand side is continuous on the set $\{(s,t,u,v,\mu,\nu)\in \mathbb{R}^4\times \mathcal{P}(M)^2: s\leq t\leq u\leq v\}.$

{\it Case 2:} $s\leq u\leq t \leq v.$
$$
\begin{aligned}
H(s,t,u,v,\mu,\nu)&=\int_{M^N}\int_{M^N}\left[\mathsf{E} {}^DK^{\otimes N}_{s,u}{}^DK^{\otimes N}_{u,t}a(x){}^DK^{\otimes N}_{u,t}{}^DK^{\otimes N}_{t,v}b(y)\right]\mu^{\otimes N}(dx)\nu^{\otimes N}(dy)\\
&=\int_{M^N}\int_{M^N}\mathsf{P}^{(N)}_{u-s}\left[\mathsf{P}^{(2N)}_{t-u}\left(a\otimes \mathsf{P}^{(N)}_{v-t}b\right)(\cdot,y)\right](x)\mu^{\otimes N}(dx)\nu^{\otimes N}(dy).
\end{aligned}
$$
We check that the right-hand side is continuous on the set $\{(s,t,u,v,\mu,\nu)\in \mathbb{R}^4\times \mathcal{P}(M)^2: s\leq u\leq t\leq v\}.$ Let $(s_n,t_n,u_n,v_n,\mu_n,\nu_n)\to (s,t,u,v,\mu,\nu),$ $s_n\leq u_n\leq t_n\leq v_n.$ Denote $G_{t,u,v}(x,y)=\mathsf{P}^{(2N)}_{t-u}\left(a\otimes \mathsf{P}^{(N)}_{v-t}b\right)(x,y)\in C_0(M^{2N}),$ where $x,y\in M^N.$ 
We have uniform estimates 
$$
\begin{aligned}
\ & \sup_{(x,y)\in M^{2N}}\left|G_{t_n,u_n,v_n}(x,y)-G_{t,u,v}(x,y)\right|\\
& \leq \sup_{(x,y)\in M^{2N}}\left|\mathsf{P}^{(2N)}_{t_n-u_n}\left(a\otimes \mathsf{P}^{(N)}_{v_n-t_n}b\right)(x,y)-\mathsf{P}^{(2N)}_{t_n-u_n}\left(a\otimes \mathsf{P}^{(N)}_{v-t}b\right)(x,y)\right|\\
& \ \ \ \ \ +  \sup_{(x,y)\in M^{2N}}\left|\mathsf{P}^{(2N)}_{t_n-u_n}\left(a\otimes \mathsf{P}^{(N)}_{v-t}b\right)(x,y)-\mathsf{P}^{(2N)}_{t-u}\left(a\otimes \mathsf{P}^{(N)}_{v-t}b\right)(x,y)\right|\\
&\leq \|a\|\|\mathsf{P}^{(N)}_{v_n-t_n}b-\mathsf{P}^{(N)}_{v-t}b\|+\|\mathsf{P}^{2N}_{t_n-u_n}\left(a\otimes \mathsf{P}^{(N)}_{v-t}b\right)-\mathsf{P}^{2N}_{t-u}\left(a\otimes \mathsf{P}^{(N)}_{v-t}b\right)\|;
\end{aligned}
$$
$$
\begin{aligned}
& \sup_{(x,y)\in M^{2N}}\left|\mathsf{P}^{(N)}_{u_n-s_n}\left[G_{t_n,u_n,v_n}(\cdot,y)\right](x)-\mathsf{P}^{(N)}_{u-s}\left[G_{t,u,v}(\cdot,y)\right](x)\right|\\
& \leq \|a\|\|\mathsf{P}^{(N)}_{v_n-t_n}b-\mathsf{P}^{(N)}_{v-t}b\|+\|\mathsf{P}^{2N}_{t_n-u_n}\left(a\otimes \mathsf{P}^{(N)}_{v-t}b\right)-\mathsf{P}^{2N}_{t-u}\left(a\otimes \mathsf{P}^{(N)}_{v-t}b\right)\|\\
& \ \ \ \ \ +\|(\mathsf{P}^{(N)}_{u_n-s_n}\otimes I)G_{t,u,v}-(\mathsf{P}^{(N)}_{u-s}\otimes I)G_{t,u,v}\|,
\end{aligned}
$$
where $I$ is the identity operator on $C_0(M^N).$ Finally,
$$
\begin{aligned}
& \left|H(s_n,t_n,u_n,v_n,\mu_n,\nu_n)-H(s,t,u,v,\mu,\nu)\right|\\
& =\bigg|\int_{M^N}\int_{M^N}\mathsf{P}^{(N)}_{u_n-s_n}\left[G_{t_n,u_n,v_n}(\cdot,y)\right](x)\mu^{\otimes N}_n(dx)\nu^{\otimes N}_n(dy) \\
& \ \ \ \ \ -\int_{M^N}\int_{M^N}\mathsf{P}^{(N)}_{u-s}\left[G_{t,u,v}(\cdot,y)\right](x)\mu^{\otimes N}(dx)\nu^{\otimes N}(dy)\bigg|\\
& \leq \|a\|\|\mathsf{P}^{(N)}_{v_n-t_n}b-\mathsf{P}^{(N)}_{v-t}b\|+\|\mathsf{P}^{2N}_{t_n-u_n}\left(a\otimes \mathsf{P}^{(N)}_{v-t}b\right)-\mathsf{P}^{2N}_{t-u}\left(a\otimes \mathsf{P}^{(N)}_{v-t}b\right)\|\\
& \ \ \ \ \ +\|(\mathsf{P}^{(N)}_{u_n-s_n}\otimes I)G_{t,u,v}-(\mathsf{P}^{(N)}_{u-s}\otimes I)G_{t,u,v}\|\\
& \ \ \ \ \ + \bigg|\int_{M^N}\int_{M^N}\mathsf{P}^{(N)}_{u-s}\left[G_{t,u,v}(\cdot,y)\right](x)\mu^{\otimes N}_n(dx)\nu^{\otimes N}_n(dy)\\
& \ \ \ \ \ -\int_{M^N}\int_{M^N}\mathsf{P}^{(N)}_{u-s}\left[G_{t,u,v}(\cdot,y)\right](x)\mu^{\otimes N}(dx)\nu^{\otimes N}(dy)\bigg|\to 0, \ n\to\infty.
\end{aligned}
$$

{\it Case 3:} $s\leq u\leq v \leq t.$
$$
\begin{aligned}
H(s,t,u,v,\mu,\nu)&=\int_{M^N}\int_{M^N}\left[\mathsf{E} {}^DK^{\otimes N}_{s,u}{}^DK^{\otimes N}_{u,v} {}^DK^{\otimes N}_{v,t}a(x){}^DK^{\otimes N}_{u,v}b(y)\right]\mu^{\otimes N}(dx)\nu^{\otimes N}(dy)\\
&=\int_{M^N}\int_{M^N}\mathsf{P}^{(N)}_{u-s}\left[\mathsf{P}^{(2N)}_{v-u}\left(\mathsf{P}^{(N)}_{t-v}a\otimes b\right)(\cdot,y)\right](x)\mu^{\otimes N}(dx)\nu^{\otimes N}(dy).
\end{aligned}
$$
Similarly to the Case 2 we get continuity of the right-hand side on the set $\{(s,t,u,v,\mu,\nu)\in \mathbb{R}^4\times \mathcal{P}(M)^2: s\leq u\leq v\leq t\}.$
%

{\it Case 4:} $u\leq v\leq s\leq t$  is identical to the Case 1.
$$
H(s,t,u,v,\mu,\nu)=\int_{M^N}\int_{M^N}\mathsf{P}^{(N)}_{t-s}a(x)\mathsf{P}^{(N)}_{v-u}b(y)\mu^{\otimes N}(dx)\nu^{\otimes N}(dy).
$$

{\it Case 5:} $u\leq s\leq v \leq t$ is identical to the Case 2.
$$
H(s,t,u,v,\mu,\nu)=\int_{M^N}\int_{M^N}\mathsf{P}^{(N)}_{s-u}\left[\mathsf{P}^{(2N)}_{v-s}\left( \mathsf{P}^{(N)}_{t-v}a\otimes b\right)(x,\cdot)\right](y)\mu^{\otimes N}(dx)\nu^{\otimes N}(dy).
$$

{\it Case 6:} $u\leq s\leq t \leq v$ is identical to the Case 3.
$$
H(s,t,u,v,\mu,\nu)=\int_{M^N}\int_{M^N}\mathsf{P}^{(N)}_{s-u}\left[\mathsf{P}^{(2N)}_{t-s}\left(a\otimes \mathsf{P}^{(N)}_{v-t}b\right)(x,\cdot)\right](y)\mu^{\otimes N}(dx)\nu^{\otimes N}(dy).
$$

We note that the function $F(s,t,u,v,\mu,\nu)=$
$$
=\begin{cases}
\int_{M^N}\int_{M^N}\mathsf{P}^{(N)}_{t-s}a(x)\mathsf{P}^{(N)}_{v-u}b(y)\mu^{\otimes N}(dx)\nu^{\otimes N}(dy), \ s\leq t\leq u\leq v,\\
\int_{M^N}\int_{M^N}\mathsf{P}^{(N)}_{u-s}\left[\mathsf{P}^{(2N)}_{t-u}\left(a\otimes \mathsf{P}^{(N)}_{v-t}b\right)(\cdot,y)\right](x)\mu^{\otimes N}(dx)\nu^{\otimes N}(dy), \ s\leq u\leq t \leq v,\\
\int_{M^N}\int_{M^N}\mathsf{P}^{(N)}_{u-s}\left[\mathsf{P}^{(2N)}_{v-u}\left(\mathsf{P}^{(N)}_{t-v}a\otimes b\right)(\cdot,y)\right](x)\mu^{\otimes N}(dx)\nu^{\otimes N}(dy), \ s\leq u\leq v \leq t,\\
\int_{M^N}\int_{M^N}\mathsf{P}^{(N)}_{t-s}a(x)\mathsf{P}^{(N)}_{v-u}b(y)\mu^{\otimes N}(dx)\nu^{\otimes N}(dy), \ u\leq v\leq s\leq t,\\
\int_{M^N}\int_{M^N}\mathsf{P}^{(N)}_{s-u}\left[\mathsf{P}^{(2N)}_{v-s}\left( \mathsf{P}^{(N)}_{t-v}a\otimes b\right)(x,\cdot)\right](y)\mu^{\otimes N}(dx)\nu^{\otimes N}(dy), \ u\leq s\leq v \leq t, \\
\int_{M^N}\int_{M^N}\mathsf{P}^{(N)}_{s-u}\left[\mathsf{P}^{(2N)}_{t-s}\left(a\otimes \mathsf{P}^{(N)}_{v-t}b\right)(x,\cdot)\right](y)\mu^{\otimes N}(dx)\nu^{\otimes N}(dy), \ u\leq s\leq t \leq v
\end{cases}
$$
is well-defined. Hence, $F$ is continuous on its domain and gives a continuous extension of $H.$ 

Consider the general case $a,b\in C(\hat{M}^N).$ Recall exhaustive sequences $(L_j: j\geq 1),$ $(\zeta_j: j\geq 1)$ introduced in the beginning of Section \ref{sec:def}. Let $a_j=a\times \zeta_j^{\otimes N},$ $b_j=b\times \zeta_j^{\otimes N},$
$$
g_j(\varkappa)=\int_{M^N} a_j(z)\varkappa^{\otimes N}(dz), \ h_j(\varkappa)=\int_{M^N} b_j(z)\varkappa^{\otimes N}(dz).
$$
There exists a continuous function $F_j:\{(s,t)\in \mathbb{R}^2: \ s\leq t\}^2\times \mathcal{P}(M)^2\to \mathbb{R},$ such that 
$$
F_j(s,t,u,v,\mu,\nu)=\mathsf{E} g_j(\mu{}^DK_{s,t})h_j(\nu{}^DK_{u,v})
$$
for $(s,t,u,v,\mu,\nu)\in D^4\times \mathcal{P}(M)^2,$ $s\leq t,$ $u\leq v.$  

Fix $\varepsilon>0,$ $T>0$ and a compact set  $\mathcal{C}\subset \mathcal{P}(M).$ There exists compact $C\subset M,$ such that 
$$
\inf_{\varkappa\in \mathcal{C}}\varkappa(C)\geq 1-\varepsilon.
$$

Continuity of the mapping $(t,z)\mapsto \Pi^{(1)}_t(z)\in \mathcal{P}(\mathcal{P}(M))$ implies that there exists a compact $L\subset M,$ such that \cite[Ch. II, Th. 6.7]{Parthasarathy}
$$
\inf_{t\in [0,2T], z\in C}\Pi^{(1)}_t(z,\{\varkappa:\ \varkappa(L)\geq 1-\varepsilon\})\geq 1-\varepsilon.
$$
If $L\subset L_j,$ we estimate for all $(s,t,u,v,\mu,\nu)\in (D\cap [-T,T])^4\times \mathcal{C}^2,$ $s\leq t,$ $u\leq v:$
$$
\begin{aligned}
\ & \left|\mathsf{E} g(\mu{}^DK_{s,t})h(\nu{}^DK_{u,v})-F_j(s,t,u,v,\mu,\nu)\right|\\
&\leq \|a\|\mathsf{E}\int_{M^N}{}^DK^{\otimes ^N}_{u,v}\left[|b|(1-\zeta^{\otimes N}_j) \right](y)\nu^{\otimes N}(dy)\\
& \ \ \ \ +
\|b\|\mathsf{E}\int_{M^N}{}^DK^{\otimes ^N}_{s,t}\left[|a|(1-\zeta^{\otimes N}_j) \right](x)\mu^{\otimes N}(dx)\\
&\leq  \|a\| \|b\| \left(2-\mathsf{E}\left(\int_{M}{}^DK_{s,t}\zeta_j(x)\mu(dx)\right)^N-\mathsf{E}\left(\int_{M}{}^DK_{u,v}\zeta_j(y)\nu(dy)\right)^N\right)\\
&\leq \|a\|\|b\|\left(2-\left(\mathsf{E}\int_{M}{}^DK_{s,t}\zeta_j(x)\mu(dx)\right)^N-\left(\mathsf{E}\int_{M}{}^DK_{u,v}\zeta_j(y)\nu(dy)\right)^N\right)\\
&\leq 2 \|a\|\|b\|\left(1-(1-\varepsilon)^{3N}\right),
\end{aligned}
$$
where the last inequality follows from relations
$$
\begin{aligned}
\mathsf{E}\int_{M}{}^DK_{s,t}\zeta_j(x)\mu(dx)& \geq\int_{C}\mathsf{E} {}^DK_{s,t}\zeta_j(x)\mu(dx) \geq \int_C \mathsf{E}{}^DK_{s,t}(x,L)\mu(dx)\\
&=\int_C \int_{\mathcal{P}(M)}\varkappa(L)\Pi^{(1)}_{t-s}(x,d\varkappa)\mu(dx)\geq (1-\varepsilon)^3.
\end{aligned}
$$
It follows that the function $F$ is uniformly continuous on $\{(s,t)\in (D\cap [-T,T])^2: \ s\leq t\}^2\times \mathcal{C}^2$ and can be continuously extended to $\{(s,t)\in [-T,T]^2: \ s\leq t\}^2\times \mathcal{C}^2.$

\end{proof}

\begin{prp}
\label{prp:approximation} For any $T> 0,$ compact $\mathcal{C}\subset \mathcal{P}(M)$ and $\varepsilon>0$ there exists $\delta>0$ such that for all $(\mu,\nu)\in \mathcal{C}^2,$ $((s,t),(u,v))\in (D\cap [-T,T])^4$ with  $s\leq t,$ $ u\leq v,$ $|s-u|\leq \delta$, $|t-v|\leq \delta,$ $d(\mu,\nu)\leq \delta,$
$$
\mathsf{P}\{\hat{d}( \mu{}^DK_{s,t},\nu{}^DK_{u,v})\geq \varepsilon\}\leq \varepsilon.
$$
\end{prp}

\begin{proof} Assume this is not true. Then for some $\varepsilon>0$ there exist sequences $-T\leq s_n\leq t_n\leq T,$ $-T\leq u_n\leq v_n\leq T,$ $(\mu_n,\nu_n)\in \mathcal{C}^2,$ such that $|s_n-u_n|\to 0,$ $|t_n-v_n|\to 0,$ $d(\mu_n,\nu_n)\to 0,$ but 
$$
\mathsf{P}\{\hat{d}(\mu_n {}^DK_{s_n,t_n},\nu_n{}^D K_{u_n,v_n})\geq \varepsilon\}\geq  \varepsilon.
$$
We may and do assume that $(s_n,t_n,u_n,v_n,\mu_n,\nu_n)\to (s,t,s,t,\mu,\mu)\in [-T,T]^4\times \mathcal{C}^2.$

Consider closed set
$$
\Delta^c_{\varepsilon}=\{(\varkappa_1,\varkappa_2)\in \mathcal{P}(\hat{M})^2:  \ \hat{d}(\varkappa_1,\varkappa_2)\geq \varepsilon\}
$$ 
and  a function 
$$
f(\varkappa_1,\varkappa_2)=(1-R \hat{d}_2((\varkappa_1,\varkappa_2),\Delta^c_\varepsilon))_+,
$$
where $\hat{d}_2((\varkappa_1,\varkappa_2),(\nu_1,\nu_2))=\hat{d}(\varkappa_1,\nu_1)+\hat{d}(\varkappa_2,\nu_2)$ and  $R>\frac{1}{\varepsilon}.$  Then $f\in C(\mathcal{P}(\hat{M})^2).$  Denote $F(s,t,u,v,\varkappa_1,\varkappa_2)=\mathsf{E}f(\varkappa_1 {}^DK_{s,t},\varkappa_2{}^D K_{u,v}).$ By the Lemma \ref{lem:1.6.6_measures} the function $F$ has a continuous extension on $\{(s,t)\in \mathbb{R}^2:s\leq t\}^2\times \mathcal{C}^2.$ We have

$$
\begin{aligned}
\varepsilon&\leq \mathsf{P}\{(\mu_n{}^DK_{s_n,t_n},\nu_n {}^DK_{u_n,v_n})\in \Delta^c_\varepsilon\}\leq \mathsf{E} f(\mu_n {}^DK_{s_n,t_n},\nu_n {}^D K_{u_n,v_n})\\
&=F(s_n,t_n,u_n,v_n,\mu_n,\nu_n)\to F(s,t,s,t,\mu,\mu), \ n\to\infty.
\end{aligned}
$$
However, 
$$
F(s,t,s,t,\mu,\mu)=\lim_{n\to\infty} F(s_n,t_n,s_n,t_n,\mu,\mu)=0,
$$
since 
$$
\hat{d}_2((\mu {}^DK_{s_n,t_n},\mu{}^D K_{s_n,t_n}),\Delta^c_\varepsilon)\geq \inf_{(\varkappa_1,\varkappa_2)\in \Delta^c_\varepsilon}\hat{d}(\varkappa_1,\varkappa_2)\geq \varepsilon>\frac{1}{R}
$$
and $f(\mu {}^DK_{s_n,t_n},\mu {}^DK_{s_n,t_n})=0.$ Obtained contradiction proves the Proposition.

\end{proof}

\subsection{Stochastic flow of kernels $(K_{s,t}:-\infty<s\leq t<\infty)$}

Proposition \ref{prp:approximation} implies that there exists a strictly increasing sequence of positive  integers $(n_j:j\geq 1),$  such that for each $j\geq 1$ and all $(s,t,u,v,\mu,\nu)\in (D\cap [-j,j])^4\times \mathcal{P}(L_j)^2$ with $s\leq t,$ $u\leq v,$ $|s-u|\leq 2^{-n_j},$  $|t-v|\leq 2^{-n_j},$ $d(\mu,\nu)\leq 2^{-n_j},$ 
$$
\mathsf{P}\{\hat{d}( \mu{}^DK_{s,t}, \nu {}^DK_{u,v})\geq 2^{-j}\}\leq 2^{-j}.
$$
Given $t\in\mathbb{R}$ define $t_j=\max\{s\in D_{n_j}:s\leq t\}.$ We note that $0\leq t-t_j<2^{-n_j},$  and  $s\leq t\Rightarrow s_j\leq t_j.$

%
%

Fix a measurable mapping $\hat{\ell} :\mathcal{P}(\hat{M})^\mathbb{N}\to \mathcal{P}(\hat{M})$ with the following property: for any sequence $\mu=(\mu_n:n\in\mathbb{N})$ in $\mathcal{P}(\hat{M}),$ $\hat{\ell}(\mu)$ is a limit point of $\mu$ \cite[L. 7.1]{Raimond_Riabov}. 
Fix $x_0\in M$ and consider measurable mappings $\Phi:E^\mathbb{N}\to E,$ 
$\Psi:M\times E^\mathbb{N}\to \mathcal{P}(M),$ 
$$
\Phi(K)(x)=\Psi(x,K)=\begin{cases}
\hat{\ell} \left((\mathfrak{p}(K_n)(x): n\geq 1)\right), \ \mbox{ if} \ \hat{\ell} \left((\mathfrak{p}(K_n)(x): n\geq 1)\right)(M)=1\\
\delta_{x_0}, \ \mbox{ otherwise}
\end{cases}.
$$
Now we have everything ready to construct the needed stochastic flow of kernels. We will use random kernels $({}^DK_{s,t}:-\infty<s\leq t<\infty, \ (s,t)\in D^2)$ constructed in the Section \ref{subsec:SFK_dyadic}.

%

%
%
%
%
%

\begin{dfn}
\label{def:SFK}
For real $s\leq t$  we define random kernels
$$
\tilde{K}_{s,t}=\Phi(({}^DK_{s_j,t_j}: j\geq 1)), \ K_{s,t}=\mathfrak{p}(\tilde{K}_{s,t}).
$$ 
\end{dfn}


In the subsequent  sections we verify that the family $(K_{s,t}:-\infty<s\leq t<\infty)$ satisfies all conditions stated in the Theorem \ref{thm:SFK}.

\subsubsection{Consistency} We check that Definitions \ref{def:SFK} and \ref{def:1.6.1} are consistent. Let $(s,t)\in D^2,$ $s\leq t.$ Then $s_{j}=s,$ $t_j=t$ for all large enough $j.$  The needed statement follows from equalities
$$
\tilde{K}_{s,t}(x)=\hat{\ell} \left(\left(\mathfrak{p}({}^DK_{s_j,t_j})(x):j\geq 1\right)\right)=\mathfrak{p}({}^DK_{s,t})(x)={}^DK_{s,t}(x).
$$
(see statement (1) of Proposition \ref{prp:properties_of_K}). In what follows we identify ${}^DK_{s,t}$ with $K_{s,t}$ for  $(s,t)\in D^2,$ $s\leq t.$

\subsubsection{Case when $s=t$} If $s=t,$ then $s_j=t_j$ for all $j\geq 1.$  Since the kernel $x\mapsto \delta_x$ is invariant under $\mathfrak{p}$ (Lemma \ref{lem:surjectivity}) and $K_{s,s}(x)=\delta_x$ (Definition \ref{def:1.6.1}), we deduce that 
$$
\tilde{K}_{s,s}(x)=\hat{\ell} \left(\left(\mathfrak{p}(K_{s_j,s_j})(x):j\geq 1\right)\right)=\delta_{x}.
$$
It follows that $K_{s,s}(x)=\delta_x$ without exceptions.

\subsubsection{Invariance under $\mathfrak{p}$} Property $\mathfrak{p}\circ \mathfrak{p}=\mathfrak{p}$ and Definition \ref{def:SFK} immediately impliy that $\mathfrak{p}(K_{s,t})=K_{s,t}.$

\subsubsection{Measurability} Mappings $(s,t,\omega)\mapsto K_{s_j,t_j}(\omega)$ are measurable. From measurability of  $\Phi$ we immediately obtain the measurability of $(s,t,\omega)\mapsto K_{s,t}(\omega)$.


\subsubsection{Convergence of approximations and distribution of $K_{s,t}$}   \begin{prp}
\label{prp:convergence} \begin{enumerate} 
\item For all real $s\leq t$ and all $\varkappa\in \mathcal{P}(M),$
$$
\lim_{j\to\infty}\varkappa K_{s_j,t_j}=\varkappa K_{s,t} \  \mbox{ a.s.}
$$

\item The law of $K_{s,t}$ coincides with with $\nu_{t-s}.$

\end{enumerate}
\end{prp}

\begin{proof} We start by showing  that for each $x\in M$ a.s. the limit $\lim_{j\to\infty}K_{s_j,t_j}(x)$  exists in $\mathcal{P}(M)$. Fix $T>0,$ such that $-T\leq s\leq t\leq T.$ If $[-T-1,T+1]\subset [-j,j]$ and $x\in L_j,$ then under conditions $(s,t,u,v)\in (D\cap [-j,j])^4,$ $s\leq t,$ $u\leq v,$ $|s-u|\leq 2^{-n_j},$ $|t-v|\leq 2^{-n_j},$ we have 
$$
\mathsf{P}\{\hat{d}(K_{s,t}(x),K_{u,v}(x))\geq 2^{-j}\}\leq 2^{-j}.
$$
We note that for all large enough $j,$ $(s_j,t_j,s_{j+1},t_{j+1})\in (D\cap [-j,j])^4.$ Further $0\leq s_{j+1}-s_j<2^{-n_j},$ $0\leq t_{j+1}-t_j<2^{-n_j}.$ Hence,
$$
\mathsf{P}\{\hat{d}(K_{s_j,t_j}(x),K_{s_{j+1},t_{j+1}}(x))\geq 2^{-j}\}\leq 2^{-j}.
$$
It follows that with probability 1 for all large enough $j,$ $\hat{d}(K_{s_j,t_j}(x),K_{s_{j+1},t_{j+1}}(x))<2^{-j}.$ In particular, the limit $\lim_{j\to\infty}K_{s_j,t_j}(x)$ a.s. exists in $\mathcal{P}(\hat{M}).$ The law of $\lim_{j\to\infty}K_{s_j,t_j}(x)$ is $\Pi^{(1)}_{t-s}(x)\in \mathcal{P}(\mathcal{P}(M)).$ So, a.s. $\lim_{j\to\infty}K_{s_j,t_j}(x)$ is concentrated on $\mathcal{P}(M).$ 

The proved convergence implies that  the distribution of $\tilde{K}_{s,t}$ coincides with $\nu_{t-s}.$ Since $\mathfrak{p}$ is a measurable presentation of $\nu_{t-s},$  $\tilde{K}_{s,t}(x)=K_{s,t}(x)$ a.s. and the distribution of $K_{s,t}$ coincides with $\nu_{t-s}.$ The first statement of the Proposition follows from Fubini's theorem.

\end{proof}

%
%

\subsubsection{Idependent increments}  Let $t^{(1)}\leq t^{(2)}\leq \ldots \leq t^{(m)}.$ Then for each $j\geq 1$ random kernels $K_{t^{(1)}_j,t^{(2)}_j},\ldots, K_{t^{(m-1)}_j,t^{(m)}_j}$ are independent (Proposition \ref{prp:properties_of_K}). Distribution of $(K_{t^{(1)},t^{(2)}},\ldots, K_{t^{(m-1)},t^{(m)}})$ in $(E^{m-1},\mathcal{E}^{\otimes (m-1)})$ is completely determined by distributions of $(K_{t^{(k)},t^{(k+1)}}(x_r): 1\leq k\leq m-1, 1\leq r\leq l),$  where $x\in M^l,$ $l\geq 1.$ Proposition \ref{prp:convergence} implies that random kernels $K_{t^{(1)},t^{(2)}},\ldots, K_{t^{(m-1)},t^{(m)}}$ are independent as well.

\subsubsection{Evolutionary property}  

\begin{prp}
\label{prp:evolutionary_property} For all real $r\leq s\leq t$ and $x\in M,$
$$
K_{r,s}K_{s,t}(x)=K_{r,t}(x) \ \mbox{ a.s.}
$$
\end{prp}

\begin{proof}

Case when  $r=s$ or $s=t$ is trivial. Assume that $r<s<t.$ Let $\mathcal{M}_j=K_{r_j,s_j}(x),$ $\mathcal{M}=K_{r,s}(x).$ Sequence $(\mathcal{M}_j:j\geq 1)$ is  independent from $K_{s,t}$ and converges a.s. to $\mathcal{M}$ (Proposition \ref{prp:convergence}).

Choose $T>0$ such that $[r,t]\subset [-T+1,T].$ Given $\alpha>0$ find compact $\mathcal{C}\subset \mathcal{P}(M),$ such that $\mathsf{P}\{\mathcal{M}\in \mathcal{C}\}>1-\alpha$  and $\mathsf{P}\{\mathcal{M}_j\in \mathcal{C}\}>1-\alpha$ for all $j.$ By Proposition \ref{prp:approximation}  there is a strictly increasing sequence of positive integers $(j_l:l\geq 1),$ such that for all $((w,z),(u,v))\in (D\cap [-T,T])^4$ with $w\leq z,$ $u\leq v,$ $|w-u|\leq 2^{-n_{j_l}},$ $|z-v|\leq 2^{-n_{j_l}},$ and all $\varkappa\in \mathcal{C},$ 
$$
\mathsf{P}\{\hat{d}(\varkappa K_{w,z},\varkappa K_{u,v})\geq 2^{-l}\}\leq 2^{-l}.
$$
It follows that for $\varkappa\in \mathcal{C},$ 
$$
\mathsf{P}\{\hat{d}(\varkappa K_{s_{j_l},t_{j_l}},\varkappa K_{s_{j_{l+1}},t_{j_{l+1}}})\geq 2^{-l}\}\leq 2^{-l}.
$$
Proposition \ref{prp:convergence} implies 
$$
\begin{aligned}
\mathsf{P}&\{\hat{d}(\varkappa K_{s_{j_l},t_{j_l}},\varkappa K_{s,t})> 2^{-l+1}\}\leq \liminf_{L\to \infty}\mathsf{P}\{\hat{d}(\varkappa K_{s_{j_l},t_{j_l}},\varkappa K_{s_{j_L},t_{j_L}})\geq 2^{-l+1}\}\\
&\leq \liminf_{L\to\infty}\mathsf{P}\left\{\hat{d}(\varkappa K_{s_{j_l},t_{j_l}},\varkappa K_{s_{j_L},t_{j_L}})\geq \sum^{L-1}_{m=l}2^{-m}\right\}\\
&\leq \liminf_{L\to\infty}\sum^{L-1}_{m=l}\mathsf{P}\left\{\hat{d}(\varkappa K_{s_{j_m},t_{j_m}},\varkappa K_{s_{j_{m+1}},t_{j_{m+1}}})\geq 2^{-m}\right\}\leq\lim_{L\to\infty}\sum^{L-1}_{m=l}2^{-m}=2^{-l+1}.
\end{aligned}
$$
When $2^{-l+1}\leq \alpha,$ we have 
\begin{equation}
\label{eq:estimate1}
\begin{aligned}
\mathsf{P}&\{\hat{d}(\mathcal{M}_{j_l}K_{s_{j_l},t_{j_l}},\mathcal{M}K_{s,t})> 2\alpha\}\\
&\leq \mathsf{P}\{\hat{d}(\mathcal{M}_{j_l}K_{s_{j_l},t_{j_l}},\mathcal{M}_{j_l}K_{s,t})> \alpha\}+ \mathsf{P}\{\hat{d}(\mathcal{M}_{j_l}K_{s,t},\mathcal{M}K_{s,t})> \alpha\}\\
&\leq \alpha+\sup_{\varkappa\in \mathcal{C}} \mathsf{P}\{\hat{d}(\varkappa K_{s_{j_l},t_{j_l}},\varkappa K_{s,t})> 2^{-l+1}\}+ \mathsf{P}\{\hat{d}(\mathcal{M}_{j_l}K_{s,t},\mathcal{M}K_{s,t})> \alpha\}\\
&\leq \alpha+2^{-l+1}+\mathsf{P}\{\hat{d}(\mathcal{M}_{j_l}K_{s,t},\mathcal{M}K_{s,t})> \alpha\}\leq 2\alpha + \mathsf{P}\{\hat{d}(\mathcal{M}_{j_l}K_{s,t},\mathcal{M}K_{s,t})> \alpha\}.
\end{aligned}
\end{equation}
By Proposition \ref{prp:approximation}  there exists $\delta>0,$  such that for all $(u,v))\in (D\cap [-T,T])^2$ with $u\leq v,$ and all $(\varkappa_1,\varkappa_2)\in \mathcal{C}^2$ with $d(\varkappa_1,\varkappa_2)\leq \delta,$
$$
\mathsf{P}\{\hat{d}(\varkappa_1 K_{u,v},\varkappa_2 K_{u,v})\geq \alpha\}\leq \alpha.
$$
From Proposition \ref{prp:convergence} it follows that for all $(\varkappa_1,\varkappa_2)\in \mathcal{C}^2$ with $d(\varkappa_1,\varkappa_2)\leq \delta,$
$$
\mathsf{P}\{\hat{d}(\varkappa_1 K_{s,t},\varkappa_2 K_{s,t})> \alpha\}\leq \alpha.
$$
We estimate 
\begin{equation}
\label{eq:estimate2}
\begin{aligned}
\mathsf{P}&\{\hat{d}(\mathcal{M}_{j_l}K_{s,t},\mathcal{M}K_{s,t})> \alpha\}\\
&\leq 2\alpha+\mathsf{P}\{d(\mathcal{M}_{j_l},\mathcal{M})>\delta\}+\sup_{\substack{(\varkappa_1,\varkappa_2)\in \mathcal{C}^2 \\ d(\varkappa_1,\varkappa_2)\leq \delta}}\mathsf{P}\{\hat{d}(\varkappa_1 K_{s,t},\varkappa_2 K_{s,t})> \alpha\}\\
&\leq 3\alpha+\mathsf{P}\{d(\mathcal{M}_{j_l},\mathcal{M})>\delta\}.
\end{aligned}
\end{equation}
Substituting  \eqref{eq:estimate2} into \eqref{eq:estimate1}, we get 
$$
\mathsf{P}\{\hat{d}(\mathcal{M}_{j_l}K_{s_{j_l},t_{j_l}},\mathcal{M}K_{s,t})> 2\alpha\}\leq 5\alpha+\mathsf{P}\{d(\mathcal{M}_{j_l},\mathcal{M})>\delta\},
$$
when $2^{-l+1}\leq \alpha.$ Since $\lim_{l\to\infty}\mathsf{P}\{d(\mathcal{M}_{j_l},\mathcal{M})>\delta\}=0,$ we deduce that 
$$
\mathcal{M}_{j_l}K_{s_{j_l},t_{j_l}}\to \mathcal{M}K_{s,t}, \ l\to\infty,
$$
in probability. On the other hand, a.s.
$$
\mathcal{M}_{j_l}K_{s_{j_l},t_{j_l}}=K_{r_{j_l},t_{j_l}}(x) \to K_{r,t}(x), \ j\to\infty
$$
(Remark \ref{rem:1.6.5}). It follows that $K_{r,t}(x)=\mathcal{M}K_{s,t}=K_{r,s}K_{s,t}(x)$ a.s. This finishes the proof of Proposition \ref{prp:evolutionary_property} and of  Theorem \ref{thm:SFK} as well.

\end{proof}


\begin{thebibliography}{10}


\bibitem{Darling} R. W. R. Darling, \textit{Constructing nonhomeomorphic stochastic flows}, American Mathematical Society, Providence, RI, 1987.

\bibitem{Hajri_WBM} H. Hajri, \textit{Stochastic flows related to Walsh Brownian motion}, Electronic Journal of Probability  \textbf{16} (2011), 1563--1599.

\bibitem{Hajri_Caglar_Arnaudon} H. Hajri, M. Caglar and M. Arnaudon, \textit{Application of stochastic flows to the sticky Brownian
motion equation}, Electronic Communications in Probability \textbf{22} (2017), no.~3, 1--10.

\bibitem{Hajri_Raimond_SPA} H. Hajri and O. Raimond, \textit{Stochastic flows and an interface SDE on metric graphs},   Stochastic processes and their applications, \textbf{126} (2016), no.~1, 33--65.



\bibitem{Howitt_Warren}  C. Howitt and J. Warren, \textit{Consistent families of Brownian motions and stochastic flows of kernels}, The Annals of Probability \textbf{37} (2009), no.~4, 1237--1272.

\bibitem{Watanabe_Ikeda} N. Ikeda and Sh. Watanabe, \textit{Stochastic differential equations and diffusion processes}, North-Holland Publishing Company, Amsterdam, Oxford, New York, Kodansha LTD, Tokyo, 1981.



\bibitem{LeJan_Raimond_ALEA} Y. Le Jan and O. Raimond, \textit{Flows associated to Tanaka's SDE}, ALEA 
\textbf{1} 
(2006),
21--34.

\bibitem{LeJan_Raimond_1}
Y. Le Jan and O. Raimond,
\textit{Flows, coalescence and noise},
The Annals of Probability
\textbf{32}
(2004),
no.~2,
1247--1315.

\bibitem{LeJan_Raimond_2}
Y. Le Jan and O. Raimond,
\textit{Flows, coalescence and noise. A correction},
The Annals of Probability
\textbf{48}
(2020),
no.~3,
1592--1595.


\bibitem{LeJan_Raimond_IBVF} Y. Le Jan and O. Raimond, \textit{Integration of Brownian vector fields} Annals of Probability, \textbf{30} (2002), no.~2, 826--873.







\bibitem{LeJan_Raimond_circle} Y. Le Jan and O. Raimond, \textit{Stochastic flows on the circle}, Probability and Partial
Differential Equations in  Modem Applied Mathematics (E. C. Waymire, J. Duan, ed.), Springer Science+Business Media. Inc., New York, pp.~151--162.




\bibitem{LeJan_Raimond} Y. Le Jan and O. Raimond, \textit{Three examples of brownian flows on $\mathbb{R}$}, 
Annales de l'I.H.P. Probabilit\' es et statistiques \textbf{50} (2014) no.~4, 1323--1346.






















\bibitem{Parthasarathy} K. R. Parthasarathy, \textit{Probability measures on metric spaces},  Academic Press, New York, London, 1967. 

\bibitem{Raimond_Riabov} O. Raimond and G. Riabov, \textit{Strong Measurable Continuous Modifications of Stochastic Flows},
Ukrainian Mathematical Journal
\textbf{75}
(2024),
1722--1757.

\bibitem{Schertzer_Sun_Swart}  E. Schertzer, R. Sun and J. Swart, \textit{Stochastic flows in the Brownian web and net}, American Mathematical Society, Providence, RI, 2014.

\bibitem{Warren} J. Warren, \textit{Sticky Particles and Stochastic Flows}, In Memoriam Marc Yor --
S\' eminaire de Probabilit\' es XLVII (C. Donati-Martin, A. Lejay, A. Rouault, ed.),  Springer International Publishing, Cham, pp.~17-36.








































\end{thebibliography}
\end{document}